\documentclass{siamart190516}

\usepackage{amsmath,amsfonts,mathtools}

\usepackage{listings}
\definecolor{MyLightGrey}{rgb}{0.95,0.95,0.95}
\definecolor{MyMediumGrey}{rgb}{0.5,0.5,0.5}
\definecolor{MyDarkGrey}{rgb}{0.3,0.3,0.3}
\usepackage{algorithmic}
\Crefname{ALC@unique}{Line}{Lines} %

\usepackage{booktabs}

\usepackage{todonotes,comment}

\usepackage{hyperref}

\DeclareMathOperator{\Tr}{trace}
\DeclareMathOperator{\diagm}{diag}
\DeclareMathOperator{\ones}{\mathbf{1}}

\title{An edge centrality measure based on the Kemeny constant\thanks{Submitted to the editors DATE}\funding{This work has been partially supported by University of Pisa’s project PRA\_2020\_61, and by GNCS of INdAM}}

\author{D.~Altafini\thanks{Dipartimento di Ingegneria dell'Energia, dei Sistemi, del Territorio e delle Costruzioni, Universit\`a di Pisa, Italy} \and D.A.~Bini\thanks{Dipartimento di Matematica, Universit\`a di Pisa, Italy} \and V.~Cutini\thanks{Dipartimento di Ingegneria dell'Energia, dei Sistemi, del Territorio e delle Costruzioni, Universit\`a di Pisa, Italy}
 \and B.~Meini\thanks{Dipartimento di Matematica, Universit\`a di Pisa, Italy} \and F.~Poloni\thanks{Dipartimento di Informatica, Universit\`a di Pisa, Italy}}

\headers{An edge centrality measure}{D. Altafini, D.A. Bini, V. Cutini, B. Meini,
and F. Poloni}
\begin{document}

\maketitle

\begin{abstract} 
A new measure $c(e)$ of the centrality of an edge $e$ in an undirected graph $G$ is introduced. It is based on the variation of the Kemeny constant of the graph after removing the edge $e$. The new measure is designed in such a way that the Braess paradox is avoided.
A numerical method for computing $c(e)$ is introduced and a regularization technique is designed in order to deal with cut-edges and disconnected graphs. Numerical experiments performed on synthetic tests and on real road networks show that this measure is particularly effective in revealing bottleneck roads whose removal would greatly reduce the connectivity of the network.
\end{abstract}

\section{Introduction}
In network analysis, several measures of the importance of an edge of a graph, having different modellistic meanings and mathematical formulations, have been introduced. For instance, in \cite{bb:matrix,estrada:book} the
communicability between two nodes $i$, $j$ of a graph $G$ is defined as the $(i,j)$-th entry in the exponential of the adjacency matrix of $G$. 
The exponential of a matrix is also at the basis of the definition of importance given in \cite{dom:edge}. Other measures based on the computation of matrix functions are introduced in \cite{bk:15}, where a parameterized node centrality measure is introduced, and in \cite{bek:13} where directed networks are analyzed. 
In \cite{ck:11} the idea of considering the variation of the Kemeny constant, when an edge is removed from a graph, is considered.

In this, paper, following \cite{ck:11}, we introduce and analyze a new definition of centrality based on a modified variation of the Kemeny constant.

Let $P\in\mathbb{R}^{n\times n}$ be the transition matrix of a finite irreducible Markov chain, let $\pi=(\pi_k)_{k=1,\ldots,n}$ be its invariant measure, so that  $\pi^T P = \pi^T$ and $\pi^T \mathbf{1}=1$, where $\mathbf{1}=(1,1,\ldots,1)^T$.
The \emph{Kemeny constant} $K(P)$ is defined as the average first-passage time from a predetermined state $i\in\{1,\ldots,n\}$ to a state $j\in\{1,\ldots,n\}$ drawn randomly according to the probability distribution $\pi$. It is a surprising but well-studied fact that this definition does not depend on $i$ ~\cite{kemeny:book}.

Given a connected undirected graph $G = (V,E)$,  where $V$ is the set of vertices and $E$ the set of edges (possibly with weights), 
denote by $A$ the associated adjacency matrix.
The Kemeny constant of the graph $G$ is defined as $K(G) := K(P)$, where $P$ is the stochastic matrix $P=D^{-1}A$, with  $D = \mathrm{diag}(d)$,  $d=A\mathbf{1}$, and $\mathrm{diag}(v)$  denotes the diagonal matrix having the entries of the vector $v$ on the diagonal.
The Kemeny constant gives a global measure of the non-connectivity of a network \cite{BH19,ck:11,ck:12}. Indeed, if $G$ is not connected then the Kemeny constant cannot be defined or, in different words, it takes the value infinity. 

 Following the idea of \cite{ck:11}, we may formally define the Kemeny-based {\em centrality score} $c(e)$ of  the edge $e$ as 
\[
c(e) := K((V, E \setminus \{e\})) - K((V, E)),
\]
i.e., the change of the connectivity of the graph measured by the Kemeny constant, when the edge $e$ is removed from the graph itself. This quantity is well defined assuming that also $(V, E \setminus \{e\})$ is connected. Recall that an edge such that $(V, E \setminus \{e\})$ is disconnected is known as a \emph{cut-edge} in graph theory.

We show that, in matrix form, the value of $c(e)$ can be given in terms of the eigenvalues of the symmetric matrix $S=D^{-\frac12}AD^{\frac12}$ and of the eigenvalues of $S+C$, where $C$ is a symmetric correction of rank 2. 
A drawback of this definition of centrality score is that there exist graphs where $c(e)$ is negative for some $e$, an elementary example is shown in Section \ref{sec:disc}. 
In the literature, this fact is known as the Braess paradox \cite{ck:11}, \cite{braess}.
Its matrix explanation is that the correction $C$ is not positive semi-definite.

To overcome this drawback we propose a modified centrality measure, which is nonnegative for any graph and for any edge $e$. The underlying idea consists in modifying the correction $C$ in such a way that the new correction $\widehat C$ is a positive semi-definite matrix of rank 1. From the model point of view, this consists in replacing the edge $e=(i,j)$ with the loops $(i,i)$ and $(j,j)$. More precisely, 
the centrality score of $e$ is modified as follows
\[
c(e) := K((V, E_{i,j})) - K((V, E)),
~~~E_{i,j}=\left( E\setminus\{e\}\right) \cup \{ (i,i),(j,j)\}.
\]
Since the eigenvalues of the matrix $S+\widehat C$ are greater than or equal to the corresponding eigenvalues of $S$, then $c(e)\ge 0$ for any $e$ and for any graph. This guarantees that the Braess paradox is not encountered.
 
This definition cannot be applied in the case where $e$ is a cut-edge, i.e., $(V,E_{i,j})$ is not connected; in fact, in this case, the definition would yield $c(e)=\infty$. To overcome this drawback, we introduce the concept of {\em regularized centrality score} $c_r(e)$, depending on a regularization parameter $r>0$. The idea is to replace the Laplacian matrix $D-A$ with the regularized Laplacian matrix $(1+r)D-A$ in the formulas that give the Kemeny constant. If $e$ is not a cut-edge, then $\lim_{r\to 0} c_r(e)=c(e)$; otherwise, if $e$ is a cut-edge then $\lim_{r\to 0} c_r(e)=\infty$; indeed, expressing the Kemeny constant in terms of the eigenvalues of $P$, one sees that it contains a term $r^{-1}$. For cut-edges, the quantity $r^{-1} - c_r(e)$ is nonnegative and has a finite limit for $r\to 0$; this suggests the following definition of a {\em filtered} Kemeny-based centrality score
 \[
 \tilde{c}_r(e):=\left\{\begin{array}{ll}
 c_r(e)&\hbox{if $e$ is not a cut-edge,}\\
r^{-1} - c_r(e)&\hbox{if $e$ is a cut-edge.}
\end{array}\right. 
\]
The modified measure defined in this way is always non-negative, and seems particularly effective in highlighting bottlenecks in road networks, or so-called {\em weak ties}~\cite{granovetter} that bridge different clusters.

% The cut-edges of the graph can be detected without necessarily applying the available algorithms for this task, as the one of \cite{TC83}. In fact, in our implementation we have applied the effective heuristic which selects as cut-edges those edges such that $c_r(e)>\frac12 r^{-1}$.

We provide efficient algorithms implementing the computation  of the score either of a single edge, or of all the edges of a graph. The main tools in the algorithm  design are the Sherman-Woodbury-Morrison formula and the Cholesky factorization of the regularized Laplacian matrix $(1+r)D-A$.
  
Our algorithms have been tested both on synthetic graphs and on graphs representing real road networks, in particular, we have considered the maps of Pisa and of the entire Tuscany.
From our numerical experiments, reported in the paper, it turned out that this measure is robust, effective, and realistic from the model point of view, moreover, its computation is sufficiently fast even for large road networks.
Comparisons with other centrality measures
from \cite{estrada:book} 
 have been performed. It turns out that our model, unlike the ones based on PageRank and Betweenness of the dual graph, succeeds in detecting bridges on the river Arno and overpasses over the railroad line as important roads in the Pisa road map.
 The edge betweenness and edge current-flow betweenness are the only two measures (among those considered) that succeed, even though only partially,
 in highlighting important bottleneck roads.
  The CPU time required for the computation of this measure is comparable with that of other betweenness-based measures on planar networks of roads.
More details concerning  applications of the Kemeny-based centrality measure to road networks can be found in \cite{pisa}.
 
The paper is organized as follows. In Section~\ref{sec:kem} we recall some properties of the Kemeny constant. In Section~\ref{sec:cm} the Kemeny-based centrality measure is introduced and a matrix analysis is performed, while in Section~\ref{sec:different} a modified definition is proposed in order to avoid the Braess paradox. The regularized and filtered centrality scores are proposed in Section~\ref{sec:reg}.  Section~\ref{sec:com} is devoted to  computational issues and numerical experiments. Conclusions are drawn in Section~\ref{sec:con}.

\section{The Kemeny constant}\label{sec:kem}
 Let $P$ be the $n\times n$ transition matrix of an irreducible finite Markov chain and let $\pi$ be its steady state vector. Denote by $K(P)$ the Kemeny constant of $P$.
We recall some properties which allow to express the Kemeny constant in terms of the trace of a suitable matrix. Such expressions will be useful 
in the analysis performed in the next sections.

\begin{lemma}[\cite{WanDJV17}]\label{lem:1}
Let $g,h\in\mathbb{R}^{n\times n}$ be column vectors with $h^Tg=1$, $h^T\ones \neq 0$, $\pi^T g \neq 0$. Then, the inverse $Z := (I-P+gh^T)^{-1}$ exists, and
%\begin{equation} \label{K1}
\[
    K(P) = \Tr(Z) - \pi^T Z \ones,
%\end{equation}
\]
independently of $g,h$.
\end{lemma}
By setting $g=\ones$, one gets the following corollary.
\begin{corollary} \label{cor:K2}
Let $h$ be a column vector with $h^T\ones=1$; then, $Z = (I-P+\ones h^T)^{-1}$ exists, and
\begin{equation} \label{K2}
    K(P) = \Tr(Z) - 1.    
\end{equation}
\end{corollary}
%\todo[inline]{FP: Si può anche dimostrare (direttamente dal Lemma 1) che se $h$ è scelto in modo solamente che $h^Tu \neq 0$ allora $K(P) = \Tr(Z) - \frac{1}{h^T u}$. Questo corrisponde a ``shiftare'' l'autovalore zero di $I-P$ non a $1$ ma a un altro reale. Potrebbe essere comodo avere questa libertà in più per maggiore stabilità numerica? Per ora sto usando 1 e ignorando il problema.}

Since $P$ is an irreducible stochastic matrix, then it has a simple eigenvalue equal to 1.
The Kemeny constant can be expressed by means of the eigenvalues different from 1, according to
the following result.

\begin{corollary}
Let $\lambda_1=1, \lambda_2, \dots, \lambda_n$ be the spectrum of $P$. Then,
\begin{equation} \label{K3}
    K(P) = \sum_{\ell=2}^n \frac{1}{1-\lambda_\ell}.
\end{equation}
\end{corollary}
\begin{proof}
Take a Jordan form $P = WJW^{-1}$ with $W_{:,1} = \ones$, $W^{-1}_{1,:} = \pi^T$, and $\operatorname{diag}(J) = (1,\lambda_2,\lambda_3,\dots,\lambda_n)$ (reordering $\lambda_2,\dots,\lambda_n$ if necessary). Then, one has
%\begin{equation} \label{Zm1}
\[
    I-P+\ones\pi^T = W (I-J+e_1e_1^T) W^{-1} = WTW^{-1},
\]
%\end{equation}
where $T$ is upper triangular with $\operatorname{diag}(T) = (1, 1-\lambda_2, 1-\lambda_3,\dots, 1-\lambda_n)$. Plugging this expression into~\eqref{K2}, we get
\[
K(P) = \Tr(Z)-1 = \Tr(WT^{-1}W^{-1}) - 1 = \Tr(T^{-1}) - 1 = \sum_{\ell=2}^n \frac{1}{1-\lambda_\ell}. %\qedhere
\]
\end{proof}

\section{A centrality measure based on the Kemeny constant}\label{sec:cm}

Given a connected undirected graph (possibly weighted) $G = (V,E)$, where $V$ denotes the set of vertices and $E$ the set of edges (possibly with weights), one can define its Kemeny constant as $K(G) := K(P)$, with $P=D^{-1}A$, where $A=(a_{i,j})$ is the adjacency matrix of the network, and $D = \diagm(d)$, $d=A\ones$. The Kemeny constant gives a global measure of the connectivity of a network; in fact, small values of the constant correspond to highly connected networks, and large values correspond to a low connectivity.

To obtain a relative measure that takes into account the importance of each edge $e = (i,j)\in E$, we can define the {\em Kemeny-based  centrality score} as 
\begin{equation}\label{eq:cs}
c(e) := K((V, E \setminus \{e\})) - K((V, E)),
\end{equation}
i.e., the change in $K$ obtained by removing the edge $e$. This quantity is well defined assuming that $(V, E \setminus \{e\})$ is still connected, that is, $e$ is not a cut-edge.

Removing one edge $e=(i,j)$ corresponds to zeroing out the entries $a_{i,j}$ and $a_{j,i}$. This leads to the new adjacency matrix 
\begin{equation}\label{eq:corr}
\widehat A=A-a_{i,j}U\begin{bmatrix}
0&1\\1&0
\end{bmatrix}U^T,\quad U = \begin{bmatrix}
    e_i & e_j
\end{bmatrix}\in \mathbb{R}^{n\times 2},
\end{equation}
where $e_i$ and $e_j$ are the $i$-th and the $j$-th columns of the identity matrix $I$, respectively.
This removal changes the transition matrix $P$ into the matrix $\widehat{P}=\widehat D^{-1}\widehat A$, where $\widehat{D} = \diagm(\hat d)$, $\hat d=\widehat{A}\ones$, that differs from $P$ only in rows $i$ and $j$ since $\hat d=d-a_{i,j}(e_i+e_j)$. Hence we have

\begin{equation}\label{eq:phat}
\widehat{P} = P + UV^T
\end{equation}
where
\begin{equation}\label{eq:phat1}
V^T=\begin{bmatrix}
s_i&0\\
0&s_j
\end{bmatrix}U^TA-a_{i,j}\begin{bmatrix}
0&(d_i-a_{i,j})^{-1}\\(d_j-a_{i,j})^{-1}&0
\end{bmatrix}U^T,
\end{equation}
with
$s_i=\frac{a_{i,j}}{d_i(d_i-a_{i,j})},~s_j=\frac{a_{i,j}}{d_j(d_j-a_{i,j})}$.

\begin{theorem} \label{updateformula}
Suppose edge $e$ is not a cut-edge. Then, for the centrality score defined in  \eqref{eq:cs} we have
\begin{equation}
    c(e) = \Tr((I-V^TZU)^{-1}V^TZ^2U),
\end{equation}
where $U=\begin{bmatrix}
e_i&e_j
\end{bmatrix}$, $V^T$ is defined in \eqref{eq:phat1}, and $Z=(I-P+\ones h^T)^{-1}$ is as in Corollary~\ref{cor:K2}.
\end{theorem}
\begin{proof}
 We have
\begin{align*}
\hat{Z} &:= (I-\hat{P}+\ones h^T)^{-1} \\ 
    &= (I-P+\ones h^T - UV^T)^{-1} \\
    &= Z + ZU (I-V^TZU)^{-1}V^TZ,
\end{align*}
where we have used \eqref{eq:phat} and in the last step the Sherman-Woodbury-Morrison matrix identity~\cite[Section 2.1.4]{GvL:book}. We now use~\eqref{K2} %and \eqref{eq:phat}, 
and write
\begin{align*}
c(e) &= K(\hat{P}) - K(P) = \Tr(\hat{Z})-\Tr(Z) = \Tr(\hat{Z}-Z)\\
    &= \Tr(ZU (I-V^TZU)^{-1}V^TZ)\\
    &= \Tr((I-V^TZU)^{-1}V^TZ^2U),
\end{align*}
using the identity $\Tr(AB)=\Tr(BA)$~\cite[Chapter 1, Exercise 5]{laub:book}.
\end{proof}

Theorem~\ref{updateformula} allows us to compute the centrality score of one edge at essentially the cost of applying the matrix $Z$ to four vectors.

\subsection{A symmetrized formulation}
Observe that $P=D^{-1}A$ is such that $D^\frac12 P D^{-\frac12}=D^{-\frac12}A D^{-\frac12}$ is a symmetric matrix having the same spectrum of $P$. Therefore, in view of Corollary \ref{cor:K2} we may write
\[
k(P)=\Tr(W)-1,\quad W=(I-D^{-\frac12}A D^{-\frac12}+D^\frac12 gh^T D^{-\frac12})^{-1}.
\]
Moreover, choosing $g$ and $h$ such that $g=\ones$, $h=\frac1{\sum_{i=1}^n d_i}d$
yields
\begin{equation}\label{eq:sf1}
k(P) = \Tr(W)-1,\quad W=(I- D^{-\frac12}A D^{-\frac12}+\frac1{\|d\|_1}D^\frac12 \ones \ones^TD^\frac12)^{-1}.
\end{equation}
In the above expression, the matrix $W$ is real symmetric.

%Now, consider the matrix $\widehat P$. 
The symmetrization of the matrix $\widehat P$ can be easily obtained in a similar manner, that is,
\begin{equation}\label{eq:sf2}
k(\widehat P) = \Tr(\widehat W)-1,\quad \widehat W=(I- \widehat D^{-\frac12}\widehat A \widehat D^{-\frac12}+\frac1{\|\widehat d\|_1}\widehat D^\frac12 \ones \ones^T\widehat D^\frac12)^{-1}.
\end{equation}
Thus, we may write
$
c(e)=\Tr(\widehat W-W)
$, where $\widehat W-W$ is a low rank symmetric matrix. This fact enables us to exploit the properties of the eigenvalues of symmetric matrices like the Courant-Fischer theorem \cite[Chapter III]{bhatia:book}.

\subsection{Disconnected networks and cut-edges}\label{sec:disc}

If $P$ is \emph{reducible}, according to our earlier definitions, say, definition~\eqref{K3},  we would get $K(P) = \infty$,  since in this case $P$ has at least two eigenvalues equal to $1$. Therefore, one cannot apply the definition of the Kemeny-based centrality score. However, we may extend this definition to reducible matrices by means of a continuity argument as follows.

Assume $P$ reducible and w.l.o.g. assume $P=\hbox{diag}(P_1,P_2,\ldots,P_q)$, where $P_\ell$, $\ell=1,\ldots,q$,  are irreducible stochastic matrices. Clearly, the matrix $P$ has eigenvalues $\lambda_1=\ldots=\lambda_q=1$, and $\lambda_\ell\ne 1$ for $\ell=q+1,\ldots,n$. Observe that the perturbed matrix $P^{(\epsilon)}:=(1-\epsilon)P+\epsilon\frac1n \ones \ones^T$ is stochastic and irreducible for any $0<\epsilon\le 1$, so that $P^{(\epsilon)}$ has only one eigenvalue $\lambda_1(\epsilon)$ equal to 1. Moreover, in view of the Brauer theorem \cite{brauer}, the remaining eigenvalues of $P^{(\epsilon)}$ are given by  $\lambda_\ell(\epsilon)=(1-\epsilon)\lambda_\ell$, $\ell=2,\ldots,n$. Therefore we have 
\[
K(P^{(\epsilon)})=\sum_{\ell=2}^n\frac1{1-\lambda_\ell(\epsilon)}=\frac{q-1}\epsilon+\sum_{\ell=q+1}^n\frac1{1-(1-\epsilon)\lambda_\ell}.
\]

Now consider the matrix $\widehat P^{(\epsilon)}=(1-\epsilon)\widehat P+\epsilon\frac1n \ones\ones^T$ where $\widehat P$ is obtained by removing the edge $(i,j)$. Assume that this edge  belongs to the block $P_s$ for some $1\le s\le q$ and that it is not a cut-edge. That is, the block $\widehat P_s$ obtained after removing the edge is still irreducible. Denote $\hat\lambda_\ell$, $\ell=1,\ldots,n$ the eigenvalues of $\widehat P$ so that 
$\hat\lambda_1=\ldots=\hat\lambda_q=1$, and $\hat\lambda_\ell\ne 1$ for $\ell=q+1,\ldots,n$. By applying once again the Brauer theorem we find that $\widehat P^{(\epsilon)}$ has only one eigenvalue $\hat\lambda_1(\epsilon)=1$,  and the remaining eigenvalues are $\hat\lambda_\ell(\epsilon)=(1-\epsilon)\hat\lambda_\ell$, for $\ell=2,\ldots,n$.  Therefore we have 
\[
K(\widehat P^{(\epsilon)})=\sum_{\ell=2}^n\frac1{1-\hat \lambda_\ell(\epsilon)}=\frac{q-1}\epsilon+\sum_{\ell=q+1}^n\frac1{1- (1-\epsilon)\hat \lambda_\ell},
\]
so that
\[
K(\widehat P^{(\epsilon)})-K(P^{(\epsilon)})=
%\sum_{\ell=q+1}^n(\frac1{1-\hat\lambda_\ell(\epsilon)}
%-\frac1{1-\lambda_\ell(\epsilon)})=
\sum_{\ell=q+1}^n\left(\frac1{1-(1-\epsilon)\hat\lambda_\ell}-\frac1{1-(1-\epsilon)\lambda_\ell}\right)
,
\]
whence
\[
\lim_{\epsilon\to 0}\left( K(\widehat P^{(\epsilon)})-K(P^{(\epsilon)})\right)=
\sum_{\ell=q+1}^n\left(\frac1{1-\hat\lambda_\ell}-\frac1{1-\lambda_\ell}\right).
\]
Now recall that the removed entries $p_{i,j}$ and $p_{j,i}$  in $\widehat P$  belong to the block $P_s$ so that the eigenvalues of $\widehat P$ different from the eigenvalues of $P$ are those of the block $P_s$, except for the eigenvalue 1. Therefore we have  
\[
\lim_{\epsilon\to 0}\left(K(\widehat P^{(\epsilon)})-K(P^{(\epsilon)})\right)=K(\widehat P_{s})-K(P_{s}).
\]

From the above arguments it is natural to extend the definition of centrality score of an edge $e$ to the case of reducible matrices as follows.

\begin{definition}
Let $P = \diagm(P_1, P_2, \dots, P_q)$, $q>1$, be such that $P_i$ are irreducible stochastic matrices. Let $i,j$ belong to the set of indices of the block $P_s$, and assume that the edge $e=(i,j)$ is not a cut-edge. The Kemeny-based centrality of the edge $e$ is defined as
\[
c(e) = K(\widehat P_s) - K(P_s),
\]
where $\widehat P_s$ is the stochastic matrix obtained from $P_s$ by removing the edge $e=(i,j)$, according to equation \eqref{eq:phat}.
\end{definition}

%\todo[inline]{FP: in realtà non è chiaro che questa sia la definizione migliore. Forse è meglio ``scalare'' ogni autovalore con la dimensione della sua componente connessa? Quindi si avrebbe $\hat{K}(P) = n^\alpha K(P)$ per una rete connessa, e $\hat{K}(P) = n_1^\alpha K(P_1) + n_2^\alpha K(P_2)$ per una rete con due componenti connesse. Giocando con il valore di $\alpha$ si può dare importanza diversa ai ``cut-edge''. Ho fatto qualche esperimento (scrivendo una funzione \texttt{planar\_barbell} che genera grafi a forma di ``bilanciere'' con un cut-edge), e provato a ragionare su casi limite come il grafo pieno e quello vuoto, ma non ho molte buone idee su come far scalare bene le cose.}

If $P$ is reducible, i.e., if the graph is disconnected, then we can identify its connected components, locate the block $P_s$ containinmg the edge $(i,j)$ and apply the above definition in order to evaluate $c(e)$. 

If $e$ is a cut-edge, then clearly  $\widehat P$ is reducible so that $K(\widehat P)=\infty$, consequently $c(e)=\infty$.
Several graph-theoretical algorithms exist in literature to compute cut-edges in a graph in time $O(n+m)$~\cite{Tarjan}, and update the connected components of a graph after removing edges~\cite{Shiloach}. However, we prefer to deal with this issue by means of the regularization technique that we will describe later on.

%; for the sake of this algorithm, we decided to recompute from scratch (in $O(n)$) the connectivity of the network $\hat{A}$ after the removal of one edge, after noting that this step was not a bottleneck in practice.

\section{A non-negative Kemeny-based centrality score}\label{sec:different}
Intuitively, one expects that the connectivity of a graph should not increase if an edge is removed from the graph. Therefore, if the Kemeny constant properly describes the non-connectivity of a graph, then it should not decrease if an edge is removed. In terms of definition of centrality score given in \eqref{eq:cs}, we expect that $c(e)\ge 0$.  Unfortunately, it is not so.   

In fact, there are cases where the Kemeny constant of a graph can decrease if an edge is removed,  like in the graph with edges $E=\{(1,2), (1,3), (2,3), (3,4)\}$, shown in Figure~\ref{fig:graph}, on the left. Its Kemeny constant is $\frac{61}{24}\approx 2.54$. Removing the edge $(1,2)$, we get the graph on the right, which has a smaller Kemeny constant, i.e., $2.5$.
That is, the centrality score of the edge $(1,2)$ in this graph is {\em negative}. This fact is known in the literature as the Braess paradox \cite{ck:11}, \cite{braess}.

\begin{figure}\begin{center}
\includegraphics[width=0.45\textwidth]{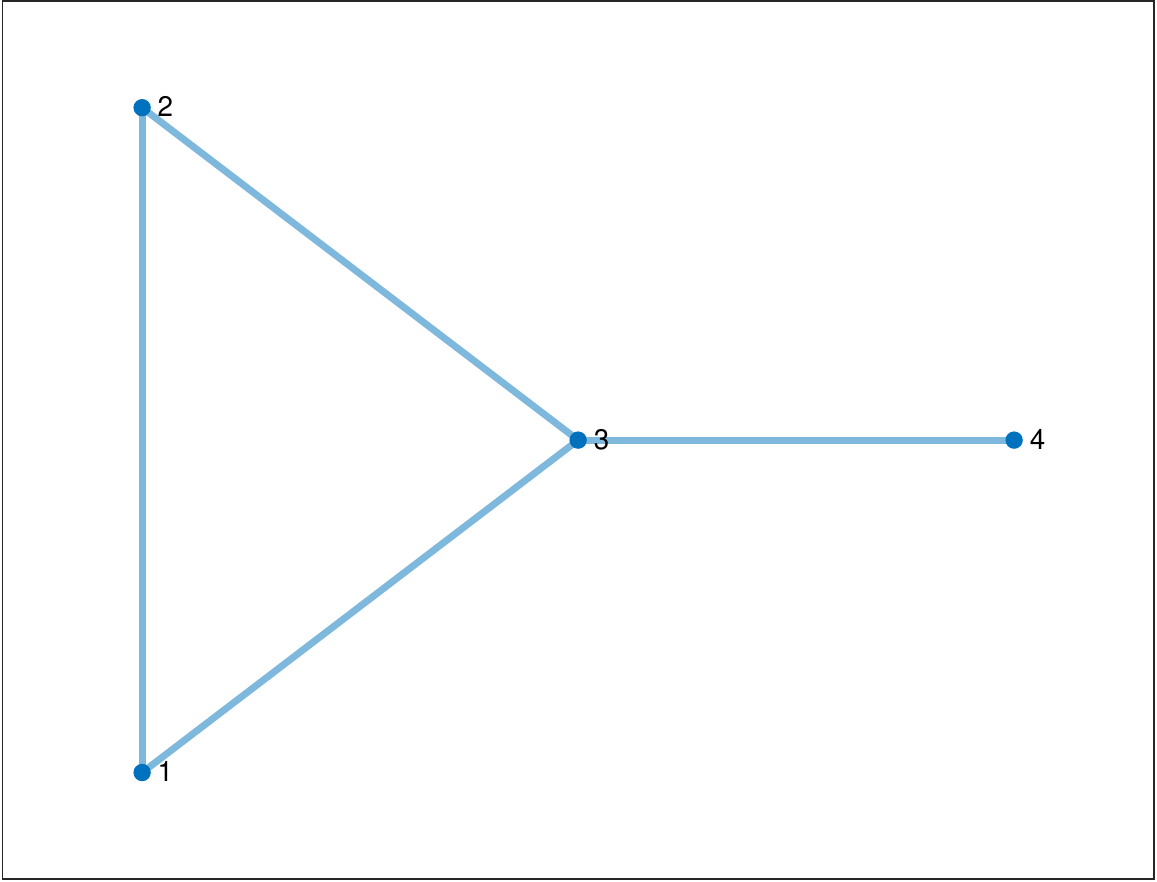}~~\includegraphics[width=0.45\textwidth]{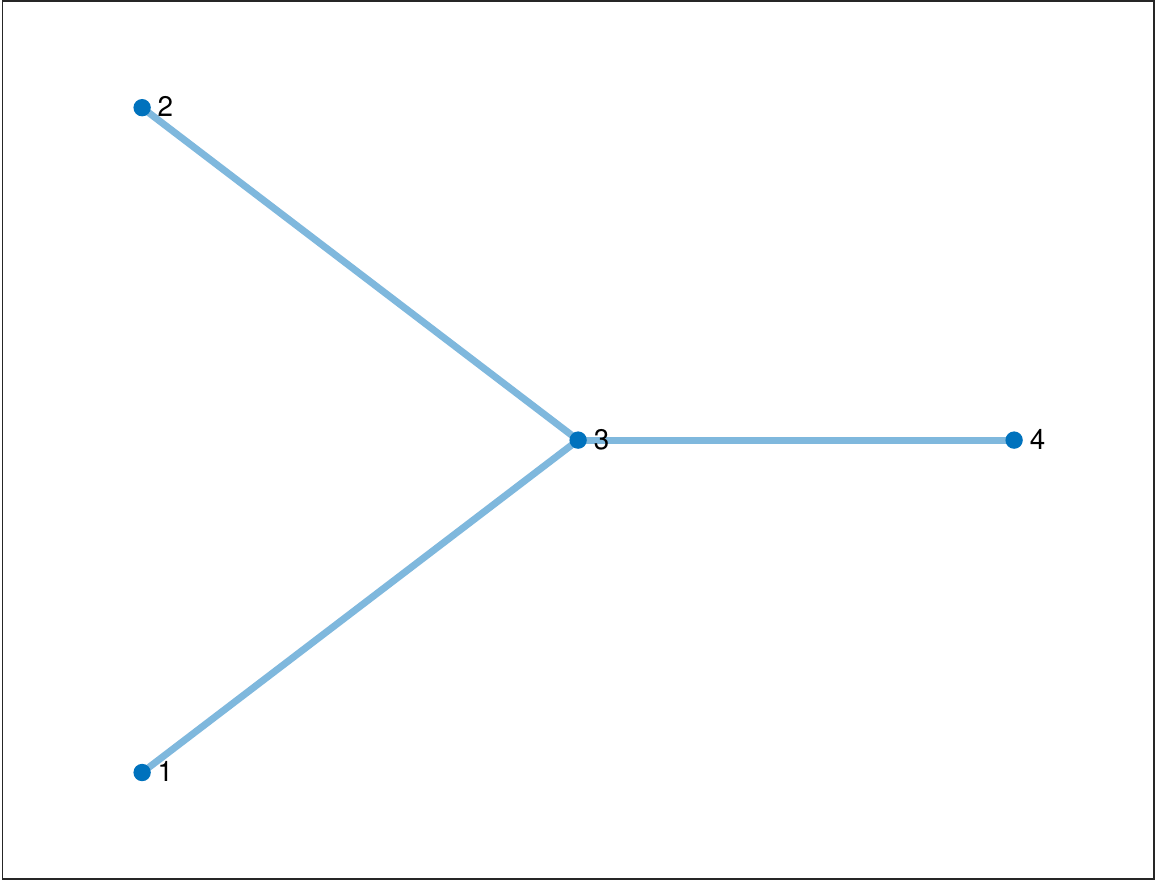}
\end{center}
\caption{\footnotesize The graph on the right is obtained from that on the left by removing the edge $(1,2)$. The two graphs have Kemeny constants $\frac{61}{24}\approx 2.54$ and $2.5$, respectively.}\label{fig:graph}
\end{figure}

In order to overcome this odd behavior of the model, where the measure $c(e)$ can take negative values, we propose a simple modification which also makes  the computation of $c(e)$ an easier task. 

Observe that removing the edge $(i,j)$ from the graph consists in performing a correction to the adjacency matrix $A$ of rank 2 in order to obtain the new matrix $\widehat A$, compare with \eqref{eq:corr}. This correction is such that the vector $d=A\ones$ differs from the vector $\hat d=\widehat A\ones$ in the components $i$ and $j$. On the other hand, defining $\widehat A$ in a different way, by means of the following expression
\begin{equation}\label{eq:corr1}
\widehat A=A+a_{i,j}vv^T,\quad v=e_i-e_j,
\end{equation}
has the effect of zeroing the entries $a_{i,j}$ and $a_{j,i}$ in $A$, and of adding $a_{i,j}$ to the diagonal entries in position $(i,i)$ and $(j,j)$. In terms of graph, this correction consists in removing the edge $(i,j)$ and adding the two loops $(i,i)$ and $(j,j)$ with the same weight $a_{i,j}$. 

The advantage of this correction is that the vectors $d=A\ones$ and $\hat d=\widehat A\ones$ satisfy the identity $\hat d= d$ since $v^T\ones=0$. This property allows us to prove that the centrality score, defined this way, always takes nonnegative values.
In order to prove this property we need to recall the following classical result that is a consequence of the Courant-Fischer minimax theorem \cite[Chapter III]{bhatia:book}.

\begin{lemma}\label{lem:eigval}
Let $A,B,C$ be real symmetric $n\times n$ matrices such that $C=A+B$, and let $\alpha_i,\beta_i,\gamma_i$, $i=1,\ldots,n$ be their eigenvalues, respectively, ordered in nondecreasing order. Then $\alpha_i+\beta_1\le\gamma_i\le\alpha_i+\beta_n$, for $i=1,\ldots,n$.
\end{lemma}

We are ready to prove the following result.

\begin{theorem}\label{th:mon} Let $A$ be the $n\times n$ adjacency matrix of an undirected graph, let
$i,j\in\{1,\ldots,n\}$ be such that the edge $e=(i,j)$ is not a cut-edge, and let $\widehat A$ be the adjacency matrix defined in \eqref{eq:corr1}.
Then for the centrality score defined as
$
c(e)=k(\widehat P)-k(P),
$
we have $c(e)\ge 0$, where $P=D^{-1}A$, $\widehat P=\widehat D^{-1}\widehat A$, $D=\widehat D=\diagm(d)$, $d=A\ones=\widehat A\ones$.
\end{theorem}

\begin{proof}
Write $c(e)$ in terms of the symmetrized formulation according to \eqref{eq:sf1} and \eqref{eq:sf2}, and get
\begin{equation}\label{eq:loc1}
c(e)=\sum_{\ell=2}^n\frac1{1-\hat \lambda_\ell}-\sum_{\ell=2}^n\frac1{1-\lambda_\ell}
\end{equation}
where $\hat\lambda_\ell$ and $\lambda_\ell$, $\ell=1,\ldots, n$ are the eigenvalues, sorted in non-increasing order, of the 
symmetric matrices 
\[
\widehat G:=\widehat D^{-\frac12}\widehat A\widehat D^{-\frac12}-\frac1{\|\widehat d\|_1}\widehat D^\frac12 \ones\ones^T \widehat D^\frac12 ,\quad\hbox{and}\quad
G:=D^{-\frac12}AD^{-\frac12}-\frac1{\|d\|_1}D^\frac12 \ones\ones^T D^\frac12,
\]
 respectively. 
On the other hand, since
$d=\hat d$ and $D=\widehat D$, we have
$\widehat G=G+a_{i,j}D^{-\frac12}vv^TD^{-\frac12}$. The matrix 
$a_{i,j}D^{-\frac12}vv^TD^{-\frac12}$ has $n-1$ eigenvalues equal to 0 and one eigenvalue equal to $a_{i,j}v^TD^{-1}v=a_{i,j}(d_i^{-1}+d_j^{-1})$. Applying Lemma \ref{lem:eigval} with $A=G$ and $B=a_{i,j}D^{-\frac12}vv^TD^{-\frac12}$ yields the inequality $\lambda_\ell\le\hat\lambda_\ell\le \lambda_\ell+a_{i,j}(d_i^{-1}+d_j^{-1})$. This implies that $c(e)\ge 0$ in view of \eqref{eq:loc1}. 
\end{proof}

Observe that $c(e)$ can be interpreted as the incremental ratio for the increment $\delta_t=1$ at $t=0$ of the function $f(t)=k(P(t))$, for $P(t)=D(t)^{-1}A(t)$, where 
$A(t)=A+ta_{i,j}vv^T$, $D(t)=\diagm(d(t))$, $d(t)=A(t)\ones$. That is,
$c(e)=\frac1{\delta_t} (f(\delta_t)-f(0))$ for $\delta_t=1$.
An interesting question is to evaluate the derivative $\lim_{\delta_t\to 0}\frac1{\delta_t}(f(\delta_t)-f(0))$ of $f(t)$ at $t=0$. We have the following result

\begin{theorem}Under the assumptions of Theorem \ref{th:mon},
let $f(t)=k(P(t))$, for $P(t)=D^{-1}A(t)$, where 
$A(t)=A+ta_{i,j}vv^T$, $D=\diagm(d)$, $d=A(t)\ones=A\ones$, $t\in[0,1]$. 
Let $\lambda_\ell$, $\ell=1,\ldots,n$ be the eigenvalues of $P(0)$, where $\lambda_1=1$. Then 
$0\le f'(0)\le
a_{i,j}(d_i^{-1}+d_j^{-1})\sum_{\ell=2}^n\frac1{(1-\lambda_\ell)^2}$.
\end{theorem}
\begin{proof}
Denote $\lambda_i(t)$ the eigenvalues of $G(t):=D^{-\frac12}A(t)D^{-\frac12}-\frac1{\|d\|_1}dd^T$, where $A(t)=A+ta_{i,j}vv^T$. We have
\[
\frac1t(k(P(t))-k(P(0))=\frac1t\sum_{\ell=2}^n
\left(\frac1{1-\lambda_\ell(t)}-\frac1{1-\lambda_\ell}\right)
=\sum_{\ell=2}^n\frac{(\lambda_\ell(t)-\lambda_\ell)/t}{(1-\lambda_\ell(t))(1-\lambda_\ell)}.
\]
Since $\lambda_\ell\le\lambda_\ell(t)\le \lambda_\ell+ta_{i,j}(d_i^{-1}+d_j^{-1})$ (compare with the proof of Theorem \ref{th:mon}), taking the limit for $t\to 0$ yields
\[
f'(t)\le a_{i,j}(d_i^{-1}+d_j^{-1})\sum_{\ell=2}^n\frac1{(1-\lambda_\ell)^2}.
\]\end{proof}

A similar inequality can be proved for $f'(1)$.
Observe that the upper bound to $f'(0)$ given in the above theorem coincides with the value $a_{i,j} (d_i^{-1}+d_j^{-1})$ up to within a constant factor independent of $i$ and $j$. 
This value depends on the out degree of node $i$ and of node $j$ independently of the topology of the graph.
%Moreover $d_i^{-1}$ is nothing else but the harmonic centrality of the node $i$, so that the value   $a_{i,j} (d_i^{-1}+d_j^{-1})$ is the sum of the harmonic centralities of the nodes that form the edge $(i,j)$ multiplied by the weight of the edge.
% This value coincides with ..?.. \cite{TODO} {\tt (verificare con la letteratura ingegneristica)}
 
Computing the value of $c(e)$ defined in \eqref{eq:corr1} is cheaper than computing the quantity defined in \eqref{eq:corr}. To this regard, we have the following

\begin{theorem}\label{th:9}
Under the assumptions of Theorem \ref{th:mon} we have
\[
c(e)=a_{i,j}v^TD^{-\frac12}W\widehat WD^{-\frac12}v
\]
where $W^{-1}=I-D^{-\frac12}AD^{-\frac12}+\frac1{\|d\|_1}D^\frac12 \ones\ones^T D^\frac12$,
$\widehat{W}^{-1}=W^{-1}-a_{i,j}D^{-\frac12}vv^TD^{-\frac12}$. Moreover,
$\widehat W=W-\tau a_{i,j}WD^{-\frac12}vv^TD^{-\frac12}W$, for $\tau=-1/(1-a_{i,j}v^TD^{-\frac12}WD^{-\frac12}v)$.
\end{theorem}
\begin{proof}
By using symmetrization, we have $c(e)=\Tr(\widehat W-W)=\Tr(\widehat W(W^{-1}-\widehat W^{-1})W)$. On the other hand, $W^{-1}-\widehat W^{-1}=a_{i,j}D^{-\frac12}vv^TD^{-\frac12}$, so that $c(e)=\Tr(a_{i,j}\widehat WD^{-\frac12}vv^TD^{-\frac12}W)= a_{i,j}v^TD^{-\frac12}W\widehat WD^{-\frac12}v$. The expression for $\widehat W$ follows from the Sherman-Woodbury-Morrison identity.
\end{proof}

The above result can be used to obtain an effective expression for computing $c(e)$. To this end, rewrite $W$ as
\[
W=D^\frac12S^{-1}D^\frac12,\quad S=D-A+\frac1{\|d\|_1}dd^T
\]
so that 
\[
\widehat W=D^\frac12\widehat S^{-1}D^\frac12,\quad 
\widehat S=D-\widehat A+\frac1{\|d\|_1}dd^T=S-a_{i,j}vv^T.
\]
Moreover, from the Sherman-Woodbury-Morrison formula we have
\[
\widehat S^{-1}=S^{-1}+\frac{a_{i,j}}{1-a_{i,j}v^TS^{-1}v}S^{-1}vv^TS^{-1}.
\]
Whence in view of Theorem \ref{th:9} we obtain
\[
c(e)=a_{i,j}v^TS^{-1}D\widehat S^{-1}v=a_{i,j}v^TS^{-1}DS^{-1}v+ \frac{a_{i,j}^2v^TS^{-1}v}{1-a_{i,j}v^TS^{-1}v} v^TS^{-1}DS^{-1}v.
\]

From the above result we obtain the following representation of $c(e)$
\begin{equation}\label{eq:rep}
\begin{aligned}
&c(e)=\frac\beta{1-\alpha},\quad\alpha=a_{i,j}v^Tx=a_{i,j}(x_i-x_j),\quad\beta= a_{i,j}x^TDx,\\
& x=S^{-1}v,\quad S=D-A+\frac1{\|d\|_1}dd^T.
\end{aligned}
\end{equation}

Observe that  $\alpha$ and $\beta$ in \eqref{eq:rep} can be rewritten as
\begin{equation}\label{eq:rep1}
\begin{aligned}
&\alpha=a_{i,j}(e_i-e_j)^TF(e_i-e_j),\quad \beta=a_{i,j}(e_i-e_j)^TQ(e_i-e_j),\\
&F=S^{-1},\quad Q=FDF.
\end{aligned}
\end{equation}
Another observation is that the matrix $S$ is positive definite since it is invertible and is the sum of two semidefinite matrices. Therefore, it admits the Cholesky factorization $S=LL^T$. 

The major computational effort in computing $c(e)$ by means of \eqref{eq:rep} consists in solving the system $Sx=v$. If one has to compute the centrality score of a single edge $(i,j)$, then two  strategies can be designed for this task. A first possibility consists in computing the Cholesky factorization of $S$ and solving the two triangular systems. This approach costs $O(n^3)$ arithmetic operations, as the dominating cost is the one of the Cholesky factorization. 
A second possibility consists in applying an iterative method for solving the linear system with matrix $S$, that exploits the low cost of the matrix-vector product, say, Richardson iteration or preconditioned conjugate gradient method. This approach costs $O(m+n)$ operations per iteration, where $m$ is the number of nonzero entries of the adjacency matrix. Thus, it is cheaper than the former approach as long as the number of required iterations is less than $n^3/m$.

A different conclusion holds in the case where the centrality scores
$c_{i,j}$ of all edges $e=(i,j)$ must be computed. In fact, in this case, the cost is $O(n^3+m)$, by relying on the following computation that is based on \eqref{eq:rep1}:
\begin{enumerate}
\item Compute  $F = S^{-1}$ and $Q=FDF$;
\item For all $i<j$ such that $a_{i,j}\ne 0$ compute:
  \begin{enumerate}
\item \quad $\alpha = a_{i,j}(r_{i,i}+r_{j,j}-2r_{i,j})$, 
\item \quad $\beta = a_{i,j}(q_{i,i}+q_{j,j}-2q_{i,j})$,
\item \quad $c_{i,j}=\beta/(1-\alpha)$.
\end{enumerate}
\end{enumerate}

The overall cost of the above approach is dominated by the cost of step 1, i.e., $O(n^3)$ arithmetic operations. The drawback of this approach is that all the $n^3$ entries of the matrices $F$ and $Q$ must be stored. This can be an issue if $n$ takes very large values.

Another issue is the potentially large condition number of the matrix $S$. A way to overcome this difficulty consists in applying a sort of regularization in the inversion of the matrix $S$. This is the subject of the next section.

\section{Regularized Kemeny-based centrality score}\label{sec:reg}
Let $r>0$ be a regularization parameter and, with the notation of the previous sections,  define the regularized Kemeny constant as 
\[
\begin{aligned}
&K_r(G)=\hbox{trace}(((1+r)I-D^{-1}A+\ones h^T)^{-1})-(1+r)^{-1}\\
&=\hbox{trace}(((1+r)I-D^{-\frac12}AD^{-\frac12}+\frac 1{\|d\|_1} D^{\frac12}\ones\ones ^TD^{\frac12})^{-1})-(1+r)^{-1},
\end{aligned}
\]
where, for the second expression, we used the symmetrized version.
 Observe that, with respect to the standard definition, we have increased the diagonal entries of the matrix $W^{-1}=I-D^{-1}A+\ones h^T$ by the quantity $r$. From one hand, this modification reduces the condition number of $W$, on the other hand, it allows to deal with the situations where $W$ is singular, for instance, in the case where the graph is not connected.
 
If the graph is connected, then $K_r(G)=\sum_{\ell=2}^n(1+r-\lambda_\ell)^{-1}$, where $1=\lambda_1>\lambda_2\ge\cdots\ge \lambda_n$ are the eigenvalues of $D^{-1}A$ ordered in a non-increasing order, moreover $K_0(G)=K(G)$. 
 On the other hand, if $G$ is not connected and is formed by two connected components, then 
 $K_r(G)=r^{-1}+\sum_{\ell=3}^n(1+r-\lambda_\ell)^{-1}$ since in this case $\lambda_1=\lambda_2=1$.

Similarly, we may define the {\em regularized Kemeny-based centrality score} of the edge $e=(i,j)$
\begin{equation}\label{eq:cre0}
c_r(e):=K_r(G\setminus\{e\})-K_r(G),
\end{equation}
so that we have
\[
c_r(e)=\sum_{\ell=2}^n((1+r-\hat\lambda_\ell)^{-1}-(1+r-\lambda_\ell)^{-1})
=\sum_{\ell=2}^n(1+r-\hat\lambda_\ell)^{-1}(\hat\lambda_\ell-\lambda_\ell)(1+r-\lambda_\ell)^{-1},
\]
where $\hat\lambda_\ell$, $\ell=1,\ldots,n$, are the eigenvalues of the matrix $D^{-1}\widehat A$, ordered in non-increasing order, for $\widehat A=A+vv^T$. Since $\hat\lambda_\ell\ge\lambda_\ell$
(compare the proof of Theorem \ref{th:mon}),
 the above equation implies that $c_r(e)\ge 0$. Observe that if $\widehat G$ is disconnected, then it is formed by two connected components and the matrix $\widehat A$ is reducible and has two eigenvalues equal to 1 so that $1=\hat\lambda_1 = \hat\lambda_2 > \hat\lambda_3$.
Thus, we may write
\begin{equation}\label{eq:cre}
c_r(e)=r^{-1}-(1+r-\lambda_2)^{-1}+\sum_{\ell=3}^n((1+r-\hat\lambda_\ell)^{-1}-(1+r-\lambda_\ell)^{-1}).
\end{equation}
In this case, we have $\lim_{r\to 0}r c_r(e) = 1$ and it turns out that the regularized centrality score of a cut-edge grows as $r^{-1}$ when $r\to 0$.
Observe also that the quantity $c_r(e)$ in \eqref{eq:cre} cannot exceed the value $r^{-1}$. In fact, we have the following result.
\begin{theorem}\label{thm:nn}
If $e$ is a cut-edge, then
for the regularized centrality score of \eqref{eq:cre0} we have
$c_r(e)\le r^{-1}$.
\end{theorem}
\begin{proof}
Since $e$ is a cut-edge, we may apply \eqref{eq:cre}, that yields
 $c(e)-r^{-1}=\sum_{\ell=3}^n[(1+r-\hat\lambda_\ell)^{-1}-(1+r-\lambda_\ell)^{-1}]-(1+r-\lambda_2)^{-1}$, where $\lambda_\ell$ and $\hat\lambda_\ell$, $\ell=1,\ldots,n$, are the eigenvalues of $D^{-\frac12}AD^{-\frac12}$ and of $D^{-\frac12}\widehat AD^{-\frac12}$, respectively, ordered in non-increasing order. Thus we get
 \begin{equation}\label{eq:tmpcre}
     c_r(e)-r^{-1}=\sum_{\ell=2}^{n-1}[(1+r-\hat\lambda_{\ell+1})^{-1}-(1+r-\lambda_\ell)^{-1}]-1/(1+r-\lambda_n).
 \end{equation}
  Since $D^{-\frac12}(\widehat A-A)D^{-\frac12}$ is a positive semidefinite rank-one matrix, from the Cauchy interlacing property~\cite[Exercise~III.2.4]{bhatia:book} we have $\hat\lambda_\ell\ge\lambda_\ell \ge \hat\lambda_{\ell+1}$ so that
$(1+r-\hat\lambda_{\ell+1})^{-1}-(1+r-\lambda_\ell)^{-1}\le 0$ which completes the proof in view of \eqref{eq:tmpcre}.
\end{proof}

Due to the additive term $r^{-1}$ in \eqref{eq:cre}, the centrality scores of the cut-edges dominate the scores of the other edges. Moreover, cut-edges that connect two large components of a graph have roughly the same score of cut-edges that connect a single node to the remaining part of the graph.
\begin{figure}
\centering
\begin{tabular}{ll}
\includegraphics[scale=0.47]{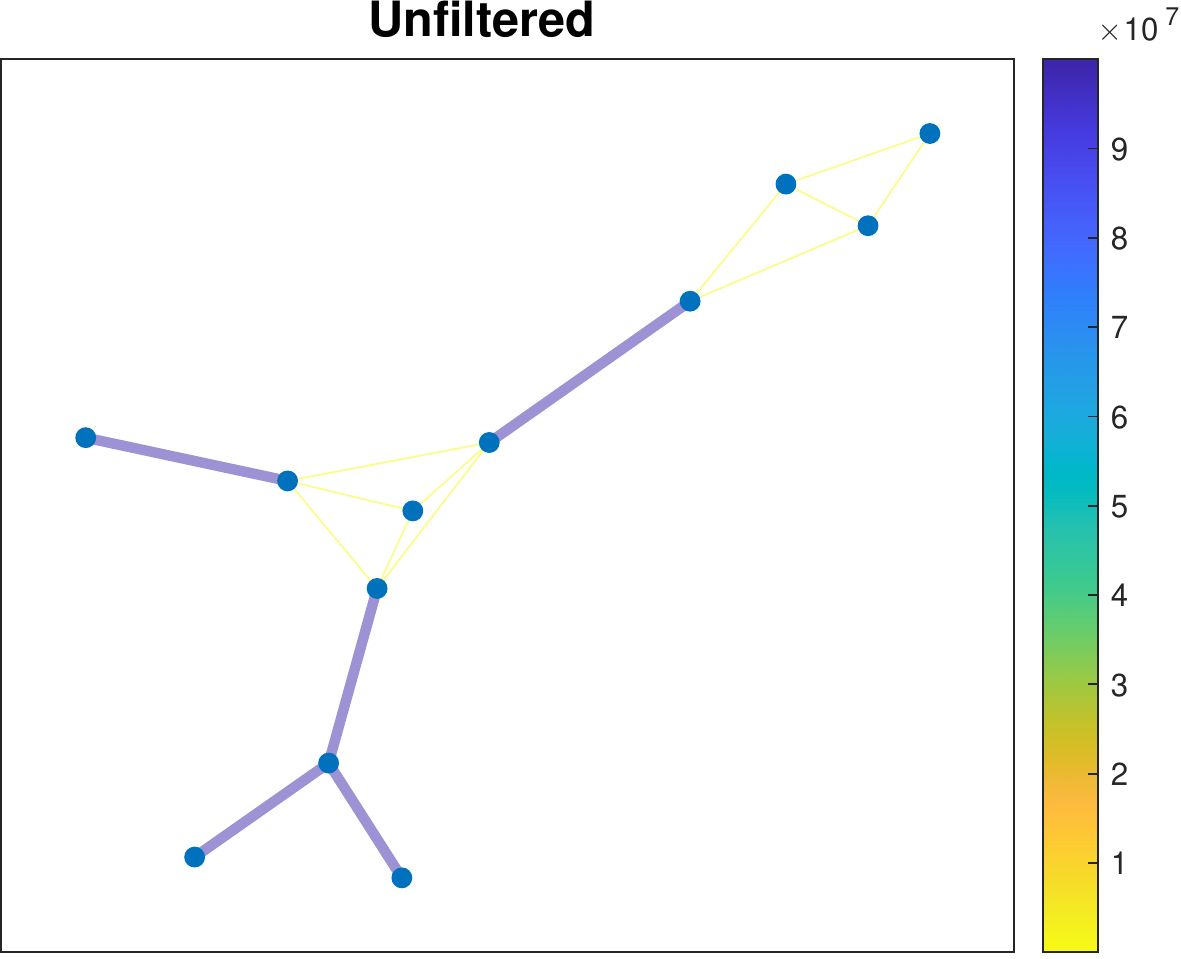} &
\includegraphics[scale=0.47]{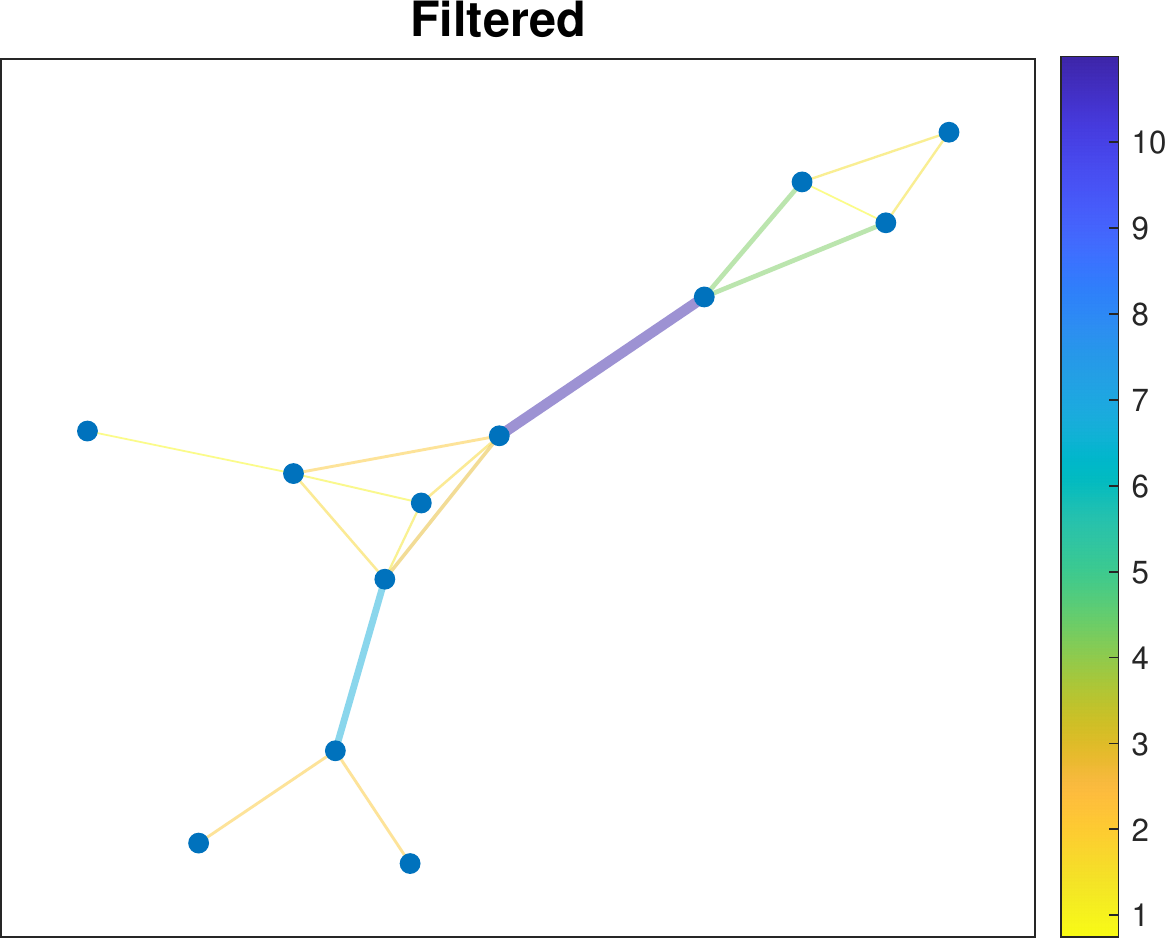}\\
\includegraphics[scale=0.47]{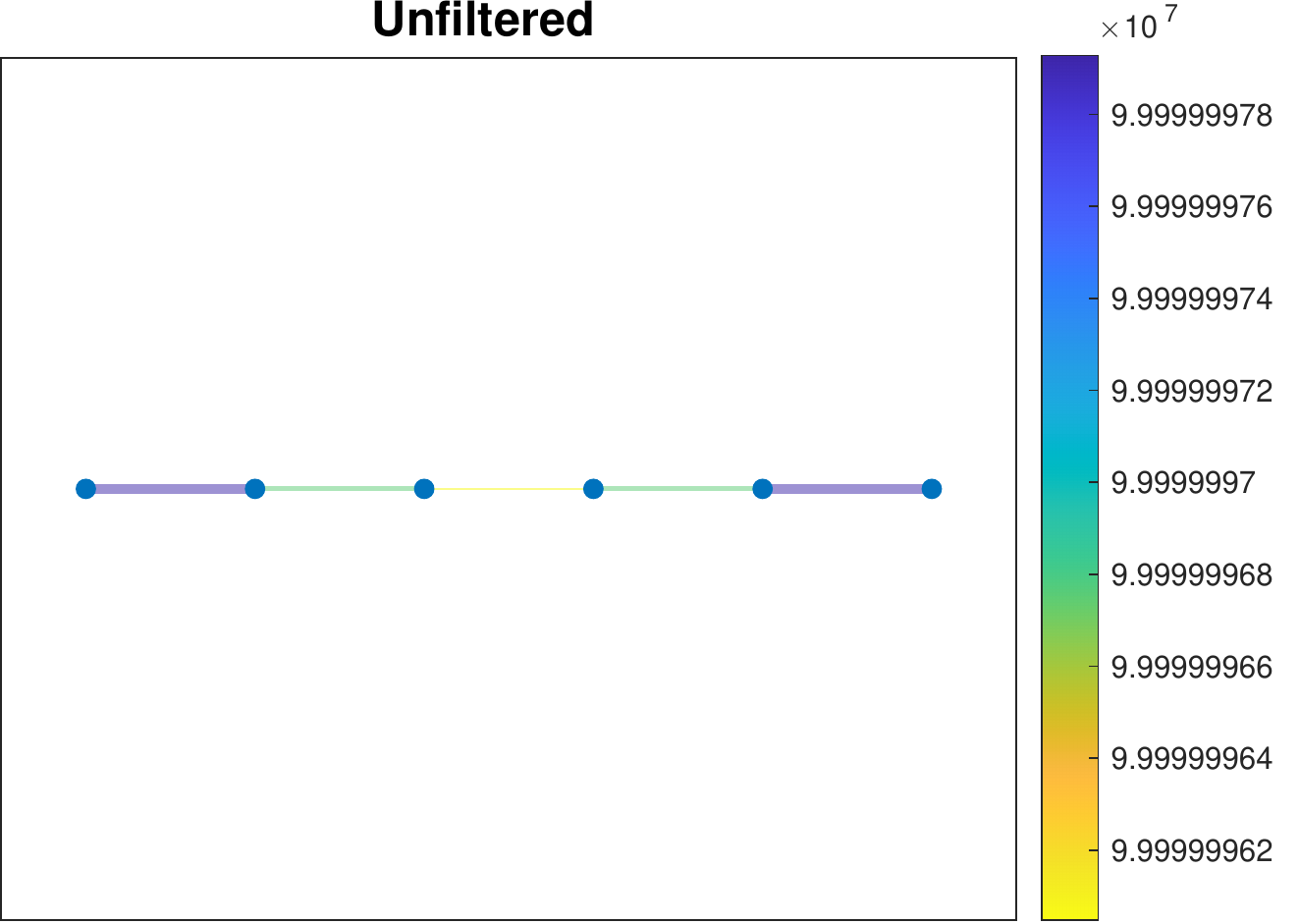} &
\includegraphics[scale=0.47]{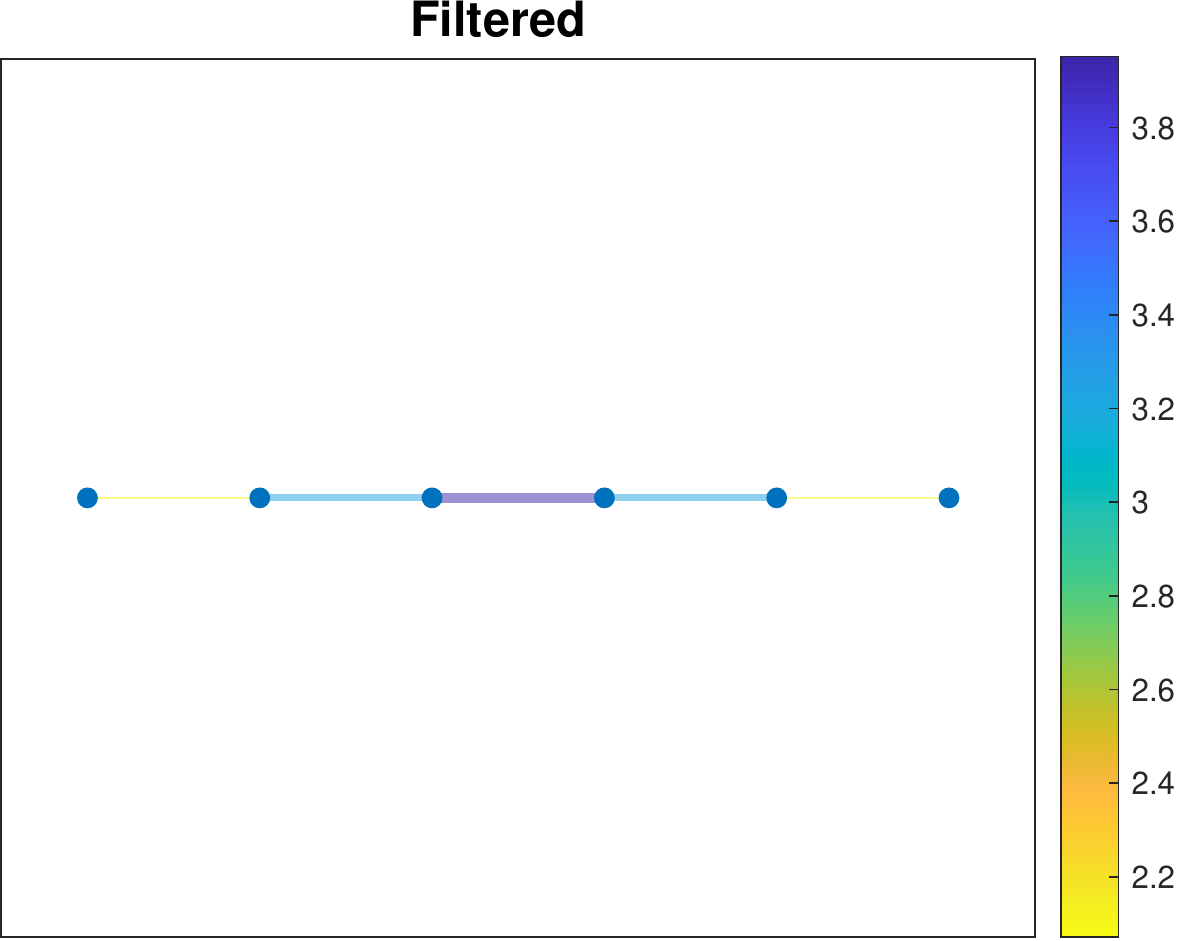}
\end{tabular}

\caption{Centrality scores of two graphs having some cut-edges: On the left the score computed by means of simple regularization with $r=10^{-8}$; on the right the score is computed by filtered regularization. Blue and thick edges denote a higher centrality.}\label{fig:graph1}
\end{figure}

%It is interesting to observe that, by denoting $G_1$ and $G_2$ the subgraphs connected by the cut-edge $e=(i,j)$, and $\widetilde G_1$, $\widetilde G_2$ the graphs obtained by adding to $G_1$ and $G_2$ the loops from $i$ to $i$ and from $j$ to $j$, then equation \eqref{eq:cre} implies that
%\[
%c_r(e)=r^{-1}+K_r(\widetilde G_1)+K_r(\widetilde G_2)-K_r(G).
%\]

A way to overcome this drawback, where any cut-edge receives a huge regularized score independently of the mass of the graphs that it connects, can be obtained by modifying definition \eqref{eq:cre0} as follows:
\begin{equation}\label{eq:ctilde}
\tilde{c}_r(e):=\left\{\begin{array}{ll}
r^{-1}-c_r(e)& \hbox{if $e$ is a cut-edge,}\\[1ex]
c_r(e)&\hbox{if $e$ is not a cut-edge,}
\end{array}\right.
\end{equation}
so that we still obtain nonnegative values in view of Theorem \ref{thm:nn}. We call $\tilde{c}_r(e)$ \emph{filtered Kemeny-based centrality}. 
From \eqref{eq:tmpcre} and \eqref{eq:ctilde}, we deduce that $\lim_{r \to 0} \tilde{c}_r(e)$ is finite; moreover,  if $e$ is not a cut-edge, then $\lim_{r \to 0} \tilde{c}_r(e)=c(e)$.

In order to figure out if $e$ is a cut-edge, one may apply the available 
computational techniques of \cite{Tarjan}, or, more simply, by selecting those edges whose regularized score is of the order of $r^{-1}$. This can be achieved by means of a heuristic strategy by computing the unfiltered values $c_r(e)$ and selecting those edges $e$ for which $c_{r}(e)>\frac12 r^{-1}$.

In Figure~\ref{fig:graph1} we report the centrality scores $c_r$ and $\tilde{c}_r$ obtained by regularization and by filtered regularization, respectively, where thick blue edges denote high centrality score.
This example shows the case  where the removal of some edge splits the graph into two components.  If $c_{r}(e)>\frac12 r^{-1}$, then $e$ is considered a cut-edge. On the left, the centrality scores with regularization parameter $r=10^{-8}$ are computed. In this case, the scores of cut-edges are of  the order of $r^{-1}$.
On the right, the filtering procedure has been applied. We may see that, in this case, only the edges that connect non-negligible subgraphs have a higher score, but not of the order of $r^{-1}$, while the remaining disconnecting edges have an intermediate moderate score.

We see that, after the filtering procedure, the values $\tilde{c}_r(e)$ have comparable magnitudes across both cut-edges and non-cut-edges, and that their ordering matches remarkably well the intuitive notion of importance of an edge for the overall connectivity of the graph.

\section{Computational issues}\label{sec:com}
In order to compute the regularized centrality $c_r(e)$ of an edge $e$, we may repeat the arguments of Section \ref{sec:different} used to provide simple formulas for computing $c(e)$. In particular, the expression of $c(e)$ given in Theorem \ref{th:9}  still holds with $W^{-1}$ and $\widehat W^{-1}$ replaced by $(rI+W)^{-1}$ and $(rI+\widehat W)^{-1}$, respectively.  This way, equations \eqref{eq:rep} are still valid with $S$ replaced by the positive definite matrix $S+r D$.
That is, instead of inverting $S$ directly, we may compute $F_r=(S+r D)^{-1}$ for a small positive value of the regularization parameter $r$. This regularization approach allows to treat also the cases where the network is disconnected, so that the matrix $S$ is singular, and the case where the removal of an edge disconnects the network. In that case, the corresponding $\alpha$ in equation \eqref{eq:rep} coincides with 1. 

The computation of the centrality score of all the edges,  with the regularization technique, is reported in Algorithm \ref{lst:kemenall}.
 %%%
\begin{algorithm}
\caption{Regularized Kemeny-based centrality of all the edges, where the number $n$ of nodes is small enough so that $n^2$ entries  can be stored in the RAM.\label{lst:kemenall}}

 % \DontPrintSemicolon{}
 % \SetKwInOut{Input}{Input}\SetKwInOut{Output}{Output}
  %%%%%%%%%%% INPUT %%%%%%%%%%%
  \hspace*{\algorithmicindent} \textbf{Input:}  The adjacency matrix $A$ and a regularizing parameter $r>0$ \\
  %\Input{The adjacency matrix $A$ and a regularizing parameter $r>0$ }
  %%%%%%%%%%% OUTPUT %%%%%%%%%%%
    \hspace*{\algorithmicindent} \textbf{Output:} The value $c_r(e)$ for any edge $e=(i,j)$
 %
 %\Output{The value $c_r(e)$ for any edge $e=(i,j)$ }
 \begin{algorithmic}[1]
  %%%%%%%%%%%%%%%%%%%%%%%%%%%%%%%%%%%
  \STATE{Compute $d=A\ones$ and $\gamma=d^T\ones $;}
  \STATE{ Set $D=\diagm(d)$ and $S=(1+r)D-A+\frac{1}{\gamma}d d^T$;}
\STATE{Compute $F=S^{-1}$, $Q=F^T D F$;}
  
  \FORALL{\rm  edge $e=(i,j)$} 
  \STATE{compute
  $\alpha = a_{ij}(f_{ii}+f_{jj}-2f_{ij})$,
           $\beta = a_{ij}(q_{ii}+q_{jj}-2q_{ij})$, and
         $c_r(e) = \beta/(1-\alpha)$.}
  \ENDFOR
  
  \end{algorithmic}
\end{algorithm}
%%%

In this approach the amount of available RAM must be of the order of $n^2$ in order to store all the entries of $S^{-1}$. Indeed, large networks require a huge storage.    

A possible way to overcome the storage  issues encountered in the case of large networks consists in exploiting the sparsity of the matrix $A$. In fact, since $T=(1+r)D-A$ is positive definite, there exists its Cholesky factorization $T=LL^T$. Moreover, the sparsity of $T$ induces a sparsity structure in $L$ so that the matrix $L$ can be stored with a low memory space and the triangular systems having matrices $L$ and $L^T$ can be solved at a low cost. An example is given in Figure \ref{fig:pisa} where the structure of $T$ and of $L$ are displayed.

\begin{figure}
\begin{center}
\includegraphics[scale=0.4]{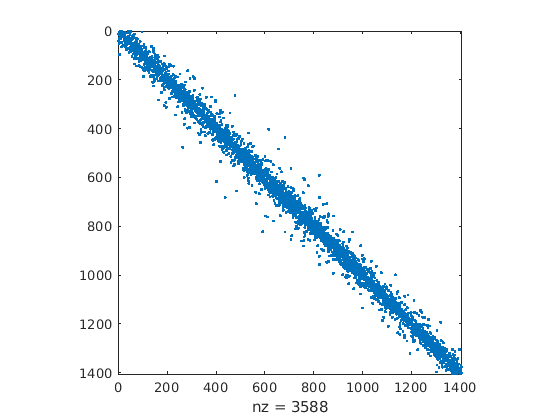}~~
\includegraphics[scale=0.4]{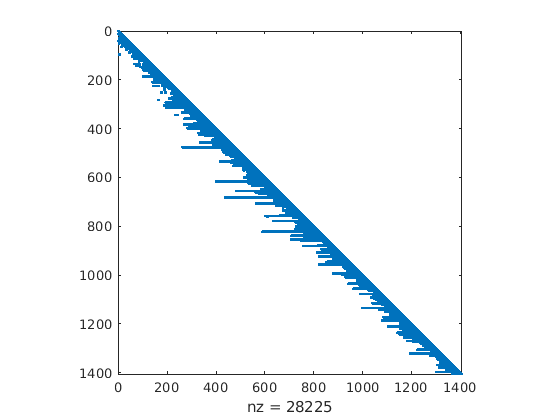}
\end{center}\caption{\footnotesize  Sparsity pattern of the adjacency matrix of the Pisa road network (left) and sparsity pattern of the Cholesky factor (right).}\label{fig:pisa}
\end{figure}

By applying the Sherman-Woodbury-Morrison identity to $S=T+\frac1{\|d\|_1}dd^T$  we may write
\[
S^{-1}=T^{-1}-\frac1{\|d\|_1+d^Tz}zz^T,\quad z=T^{-1}d
\]
so that for the vector $x$ in equation \eqref{eq:rep} we have
\[
x = w-\frac{z^Td}{\|d\|_1+d^Tz}z,\quad w=T^{-1}v.
\]
Moreover, from the Cholesky factorization $T=LL^T$ we get
\[
LL^Tz=d,\quad LL^Tw=v.
\]
The above expressions can be used together with the first equation in \eqref{eq:rep} in order to compute $c_r(e)$.

Observe that from the computational point of view, one has to compute the Cholesky factorization once for all, this is cheaper than inverting a matrix. Moreover, two sparse triangular systems with matrix $L$ and $L^T$ must be solved once for all for computing $z$. Finally, for any edge $(i,j)$, two sparse triangular systems must be solved for computing $w$ and $O(n)$ additional operations must be performed. Indeed, in this approach the cost is higher but this allows one to deal with large networks even if the amount of RAM storage is not sufficiently large. 
Algorithm \ref{lst:kemenchol} implements this approach, including regularization.

%%%
%%%
\begin{algorithm}
\caption{Regularized and filtered Kemeny-based centrality of all the edges, relying on the Cholesky factorization.\label{lst:kemenchol}}
 %%%%%%%%%%% INPUT %%%%%%%%%%%
  \hspace*{\algorithmicindent} \textbf{Input:}
  The adjacency matrix $A$ and a regularizing parameter $r>0$ \\
    %%%%%%%%%%% OUTPUT %%%%%%%%%%%
    \hspace*{\algorithmicindent} \textbf{Output:}
The value $\tilde c_r(e)$ for any edge $e=(i,j)$ 
  %%%%%%%%%%%%%%%%%%%%%%%%%%%%%%%%%%%
  \begin{algorithmic}[1]
  \STATE{Compute $d=A\ones$;}
\STATE{Set $D=\diagm(d)$ and $T=(1+r)D-A$;}
\STATE{Compute the Cholesky factorization $T=L L^T$; }
\STATE{ Solve the linear systems $L y = d$ and $L^T z=y $; }
\STATE{ Compute $\gamma=d^T z+d^T\ones $; }
\FORALL{\rm edge $e=(i,j)$}
\STATE{ set $v=e_i-e_j$} 
\STATE{  solve the systems
  $L y=v$ and $L^T w=y$}
\STATE{  set $\delta=d^T w$ and $x=w-\frac{\delta}{\gamma}z$ }
\STATE{  compute $\alpha = a_{ij}(x_i-x_j)$, $\beta = a_{ij}\sum_{\ell=1}^n x_\ell^2 d_\ell$, and $c_r(e) = \beta/(1-\alpha)$; }
\IF{$c_r(e)>\frac12 r^{-1}$}
\STATE{$\tilde c_r(e)=r^{-1}- c_r(e)$}
\ELSE
\STATE{$\tilde c_r(e)=c_r(e)$}
\ENDIF
\ENDFOR
  
  \end{algorithmic}
\end{algorithm}
%%%

Note that after the precomputation steps the computation of the centrality of each edge is independent of the others, hence the main loop can be performed in parallel.

As an example of application, we consider the cases of two barbell-shaped graphs, together with their disjoint union. The Kemeny-based centrality obtained by our approach 
%by separately considering the two graphs are reported in Figure \ref{fig:barbell-1}, while the  centrality obtained 
by considering the disjoint union of the graphs where the adjacency matrix is reducible, are reported in Figure \ref{fig:union}.

We can see from this representation that the centrality scores of the disjoint union does not differ much from the union of the centralities of the two graphs. In fact, observe  that in the rightmost graph, where the barbell is formed by two loops, the edges in the loops have a high value of centrality score. In fact, removing one of these edges almost disconnects the loop. Whereas, in the leftmost graph, where the loops are replaced by highly connected set of nodes, these edges have a low score. Indeed, their removal does not alter much the overall connectivity of the graph. On the other hand, removing one of the two edges connecting the two groups of nodes strongly reduces the connectivity between the two groups.

\begin{figure}
\begin{center}
\includegraphics[scale=0.5]{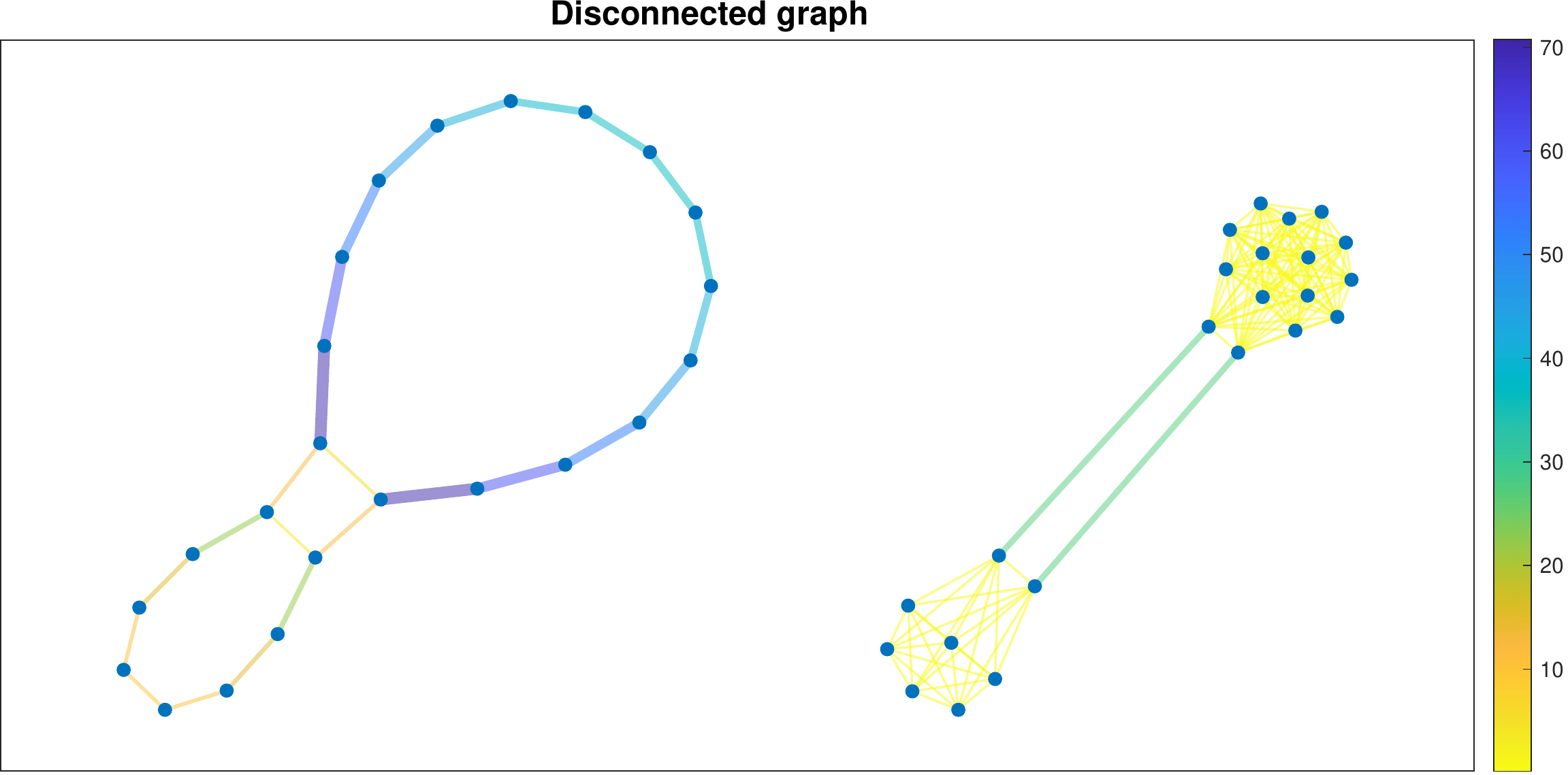}
\end{center}
\caption{\footnotesize  In this figure, the Kemeny-based centrality algorithm with regularization is applied to the reducible adjacency matrix. Blue and thick edges denote higher Kemeny-based centrality. The lack of irreducibility is overcome by the regularization technique. The values of the centralities obtained this way agree with those computed by separately applying the algorithm to the two graphs. }\label{fig:union}
\end{figure}

In the next example (Figures~\ref{fig:pisa1} and~\ref{fig:pisa2}) we consider a network composed of the roads in the city centre of Pisa. To each edge $(i,j)$ we have assigned a connection strength $w_{ij} = \exp(-\ell(i,j) / \ell_{\max}) \in (0,1]$, where $\ell(i,j)$ is the length of the edge, i.e., the Euclidean distance between points $i$ and $j$, and $\ell_{\max}$ is the maximum edge length in the network. This network is a planar undirected graph with 1794 nodes and 3240 edges; it includes many dead ends that are cut-edges, and various roads that, while not being cut-edges, are important bottlenecks for connectivity; among them are bridges on the Arno river and overpasses over the railroad line. We can see that the filtered version of the Kemeny-based centrality (here computed with regularization parameter $r=10^{-8}$) does an excellent job at highlighting these bottlenecks.

In the unfiltered version of the measure, instead, cut-edges take a very high value and are essentially the only ones to be displayed in blue. This confirms that the filtering procedure is necessary to obtain sensible results.

In the other subfigures, we display various other centrality measures that have been computed either with Matlab's \texttt{centrality} command or with Python's Networkx library~\cite{networkx}. We refer to~\cite{estrada:book} for their definitions. 
\begin{itemize}
    \item The road-taking probability in the Pagerank model (with $\alpha=0.85$) is defined for an edge $(i,j)$ as $rt((i,j)) = \pi_i R_{ij} + \pi_j R_{ji}$, where $\pi$ is the Pagerank vector and $R=\alpha P + (1-\alpha)\frac{1}{n}\mathbf{1}\mathbf{1}^T$ is the stochastic transition matrix of the Pagerank model; this quantity corresponds to the long-term probability that a random surfer goes through that edge (in any direction).
    \item Pagerank and Betweenness on the dual graph are defined using the so-called \emph{dual graph}, or \emph{line graph}, of the network; i.e., a graph in which each road is a node, and two nodes are joined by an edge if the corresponding roads meet. This allows us to compute edge centrality measures with these two algorithms, which were designed to compute node importance.    
    \item Edge betweenness (with $\ell(i,j)$ as the distance) and edge current-flow betweenness~\cite{bf:05} (with resistances $w_{ij}$) are computed using Python's Networkx library, while all previous measures were computed with Matlab's \texttt{graph/centrality} command. These are the only other two measures (among those considered) that partially do a similar job of highlighting bottleneck edges.
\end{itemize}
\begin{figure}
\centering

\begin{tabular}{ll}
\includegraphics[width=0.42\textwidth]{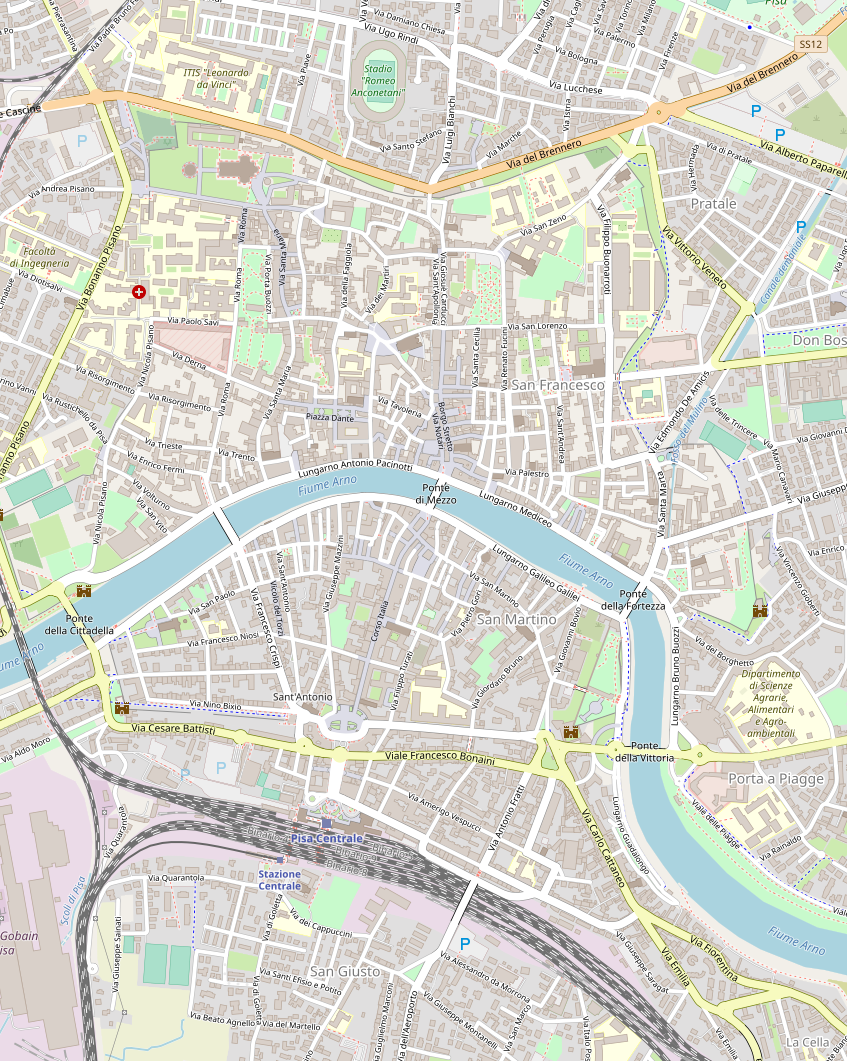} &
\includegraphics[width=0.45\textwidth]{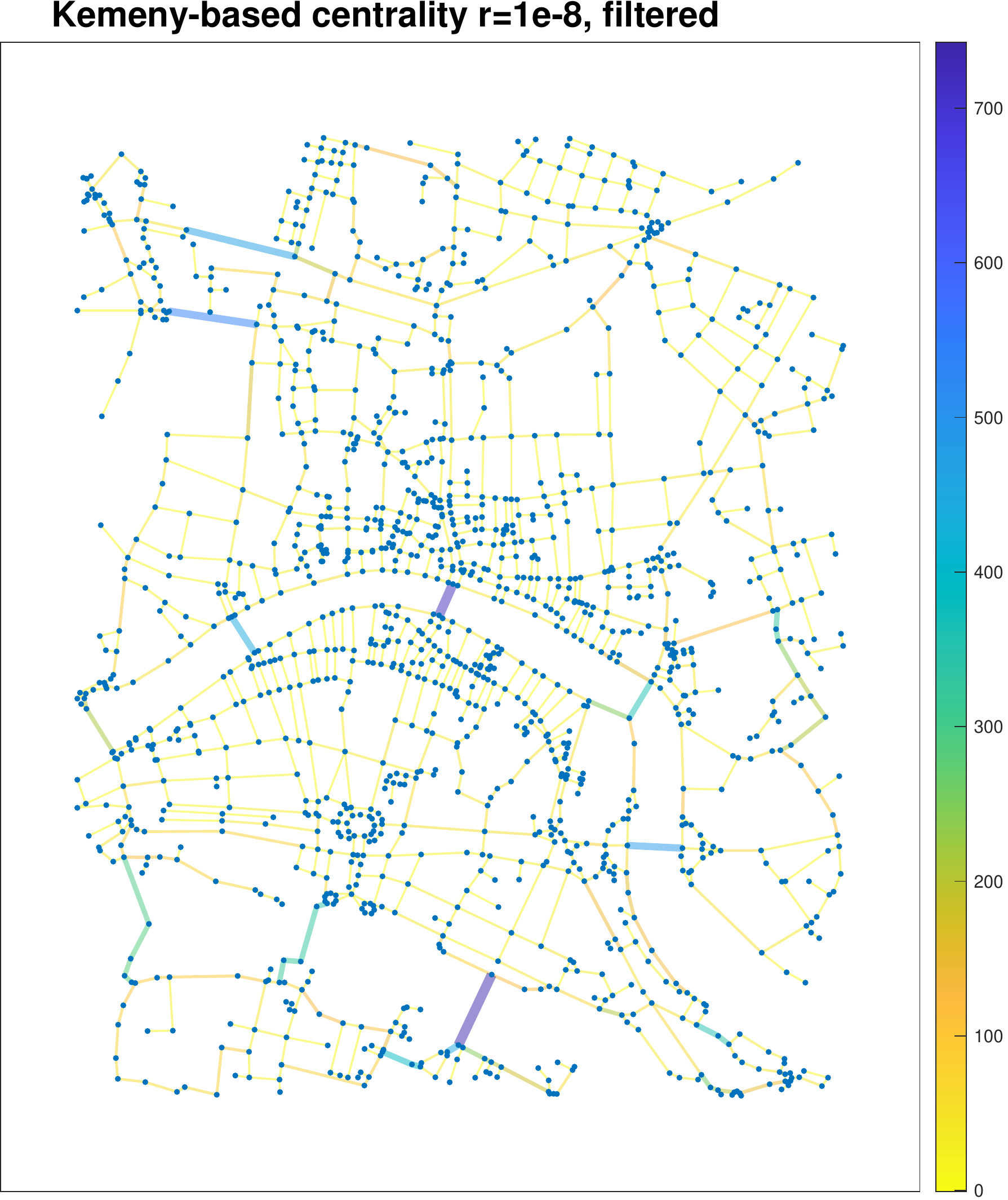}\\
\includegraphics[width=0.45\textwidth]{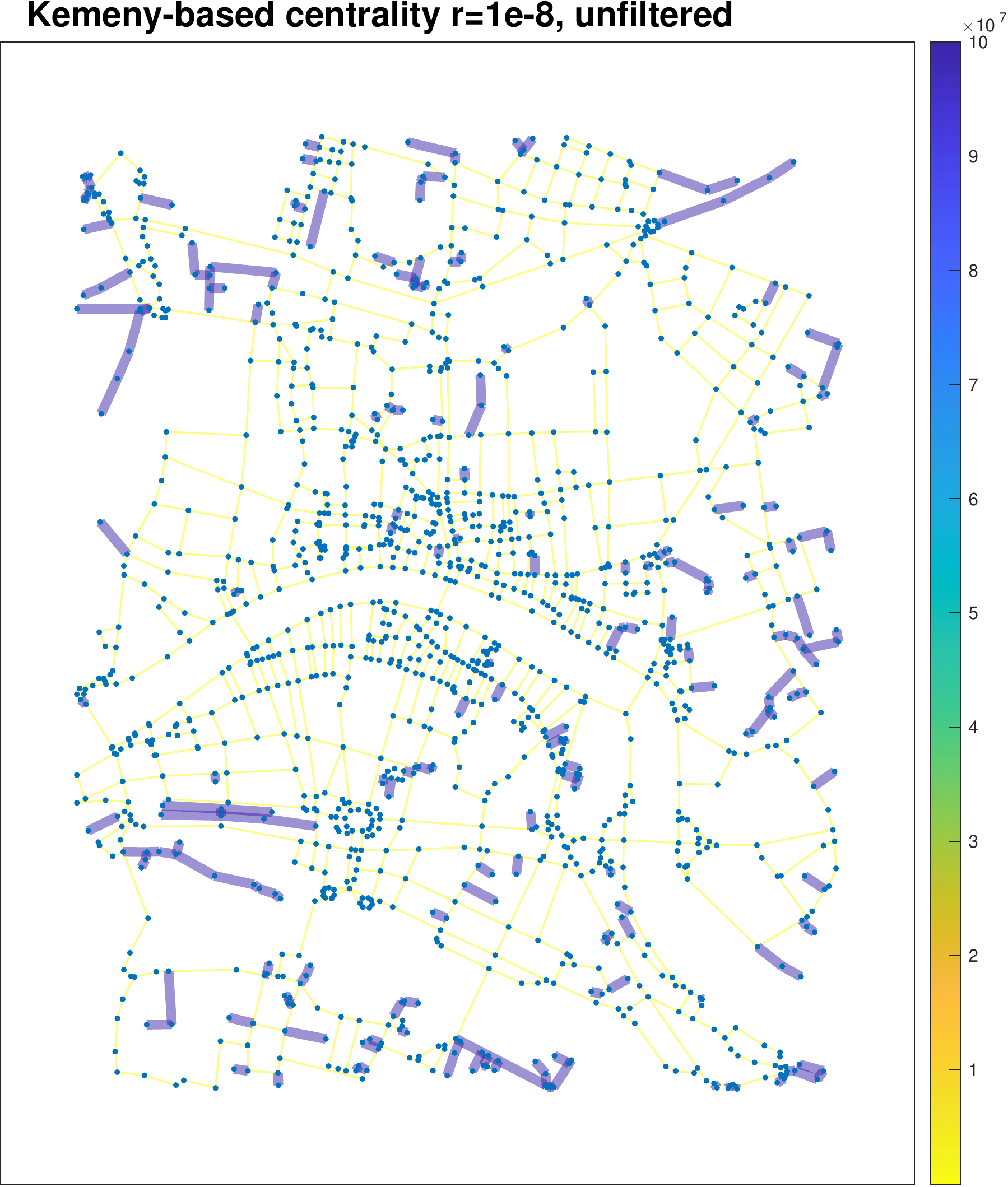} &
\includegraphics[width=0.45\textwidth]{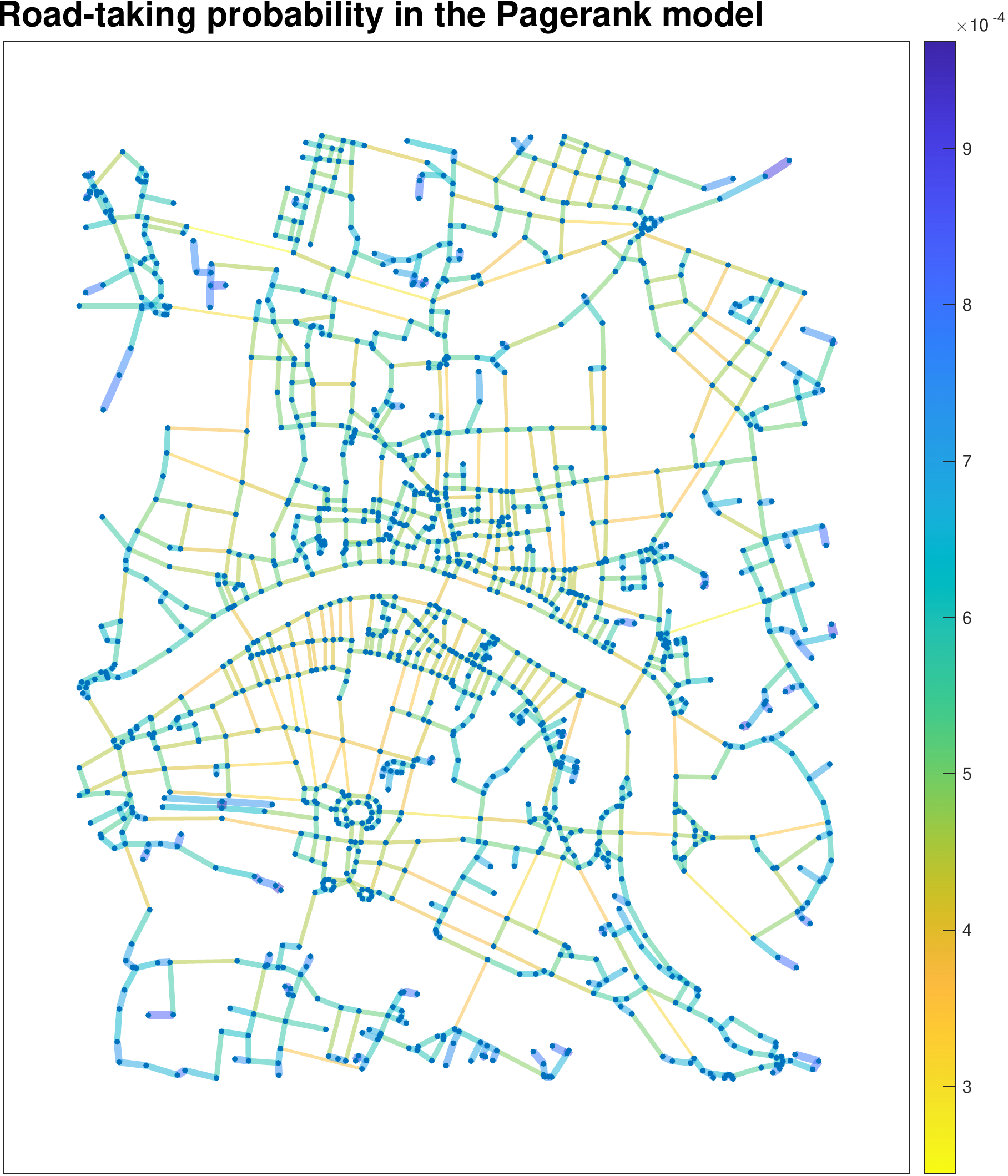}
\end{tabular}

\caption{Comparison of several centrality measures on a map of the Pisa city center; part I.} \label{fig:pisa1}
\end{figure}
\begin{figure}

\begin{tabular}{ll}
\includegraphics[width=0.45\textwidth]{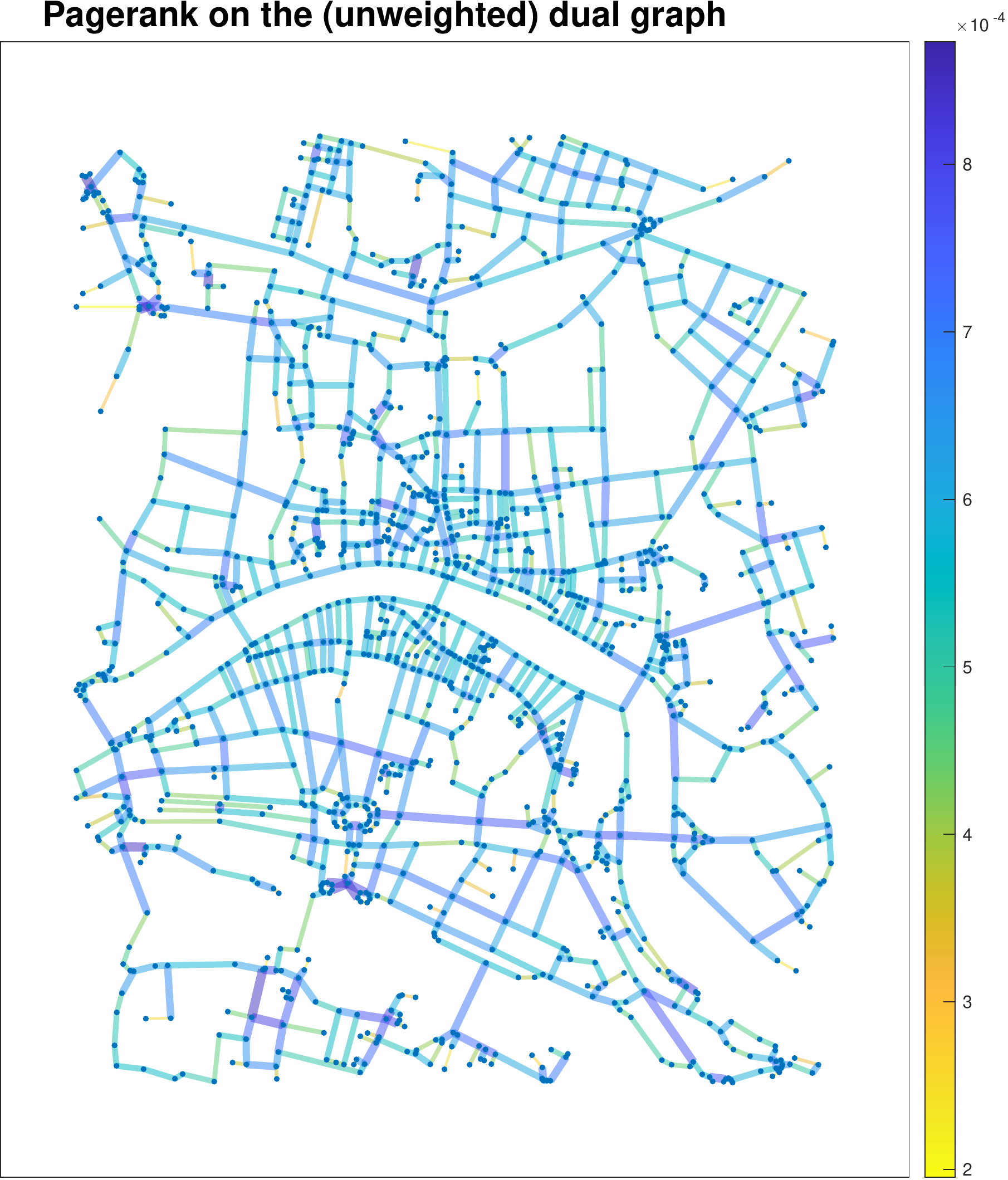} &
\includegraphics[width=0.45\textwidth]{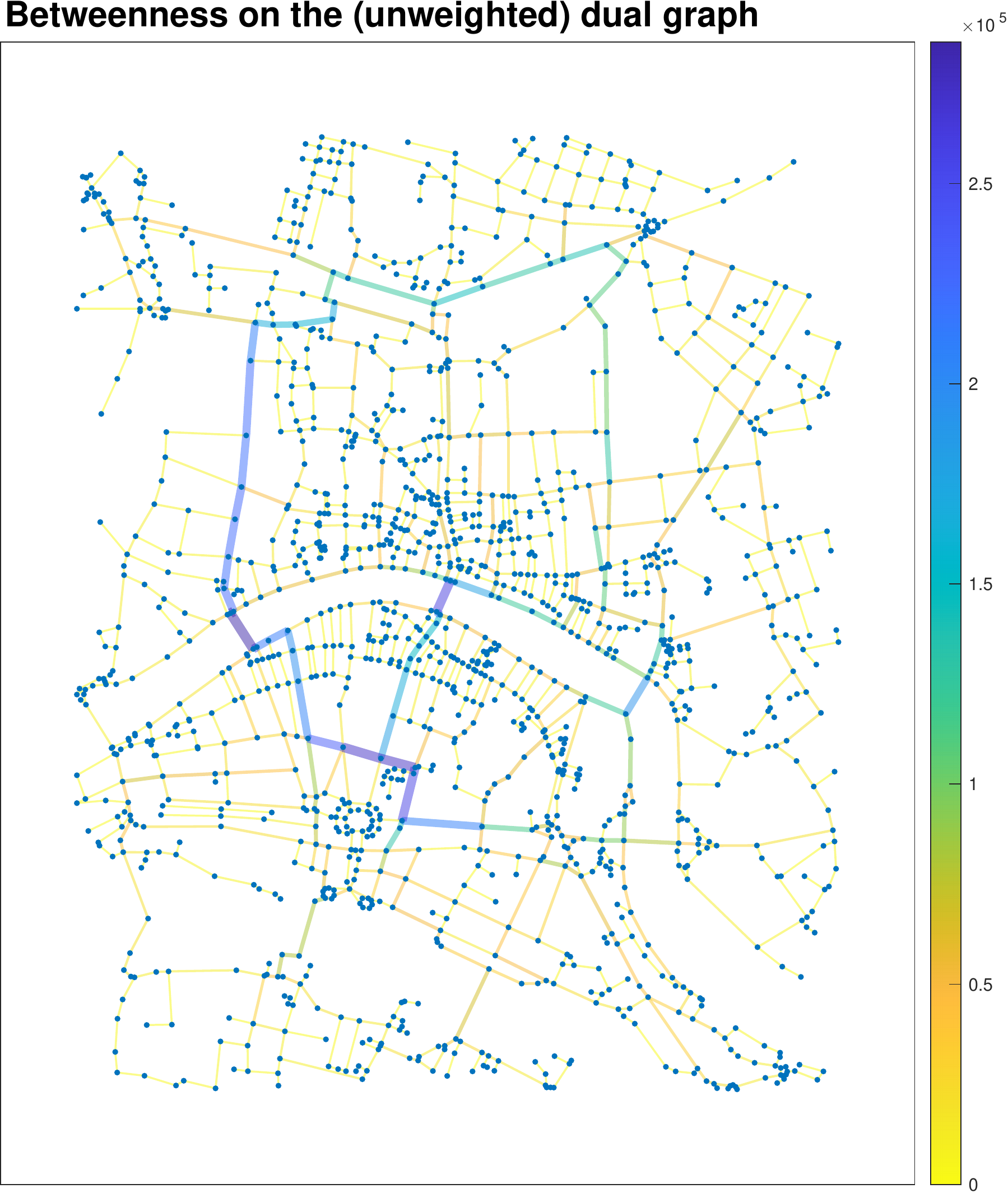}\\
\includegraphics[width=0.45\textwidth]{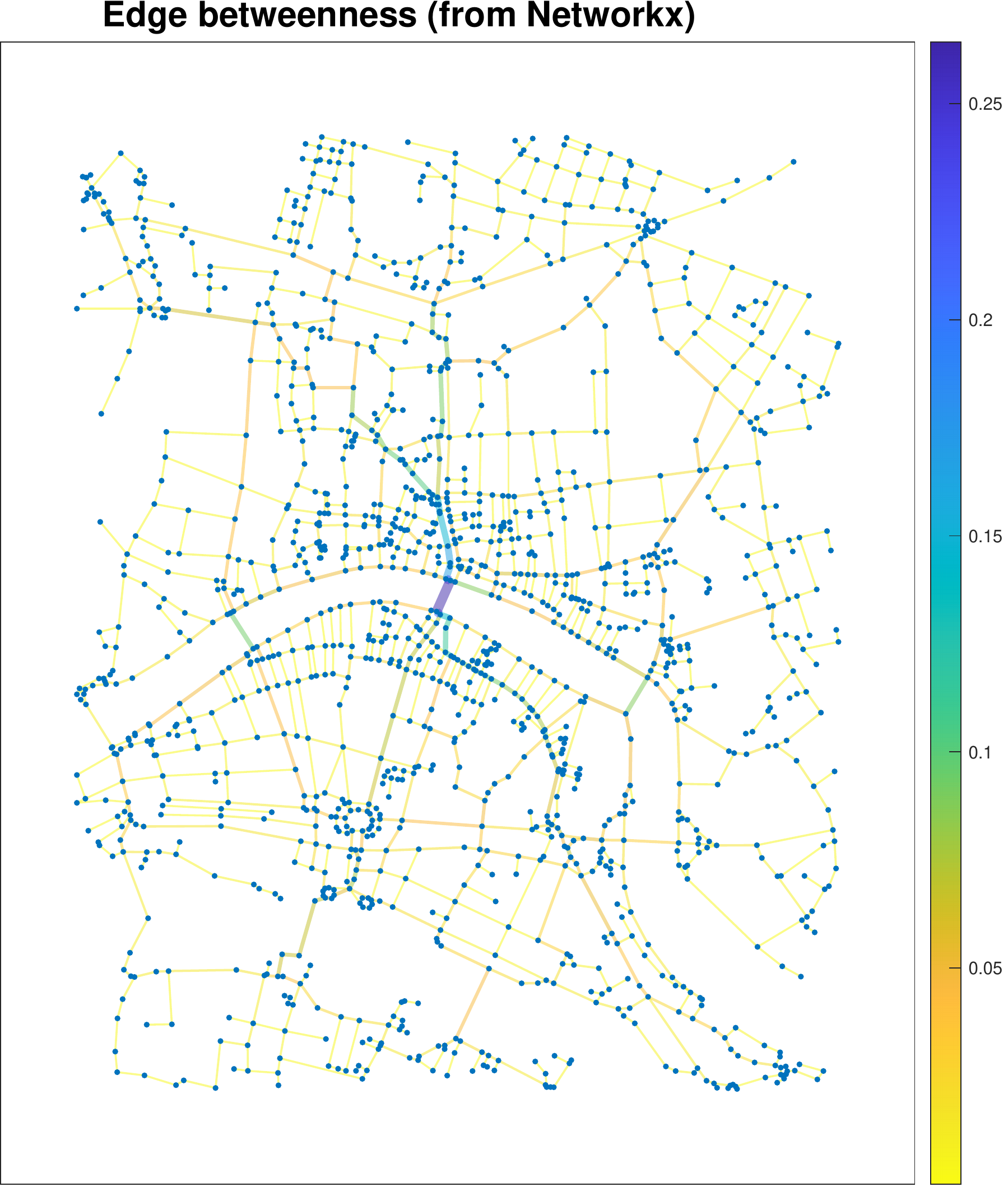} &
\includegraphics[width=0.45\textwidth]{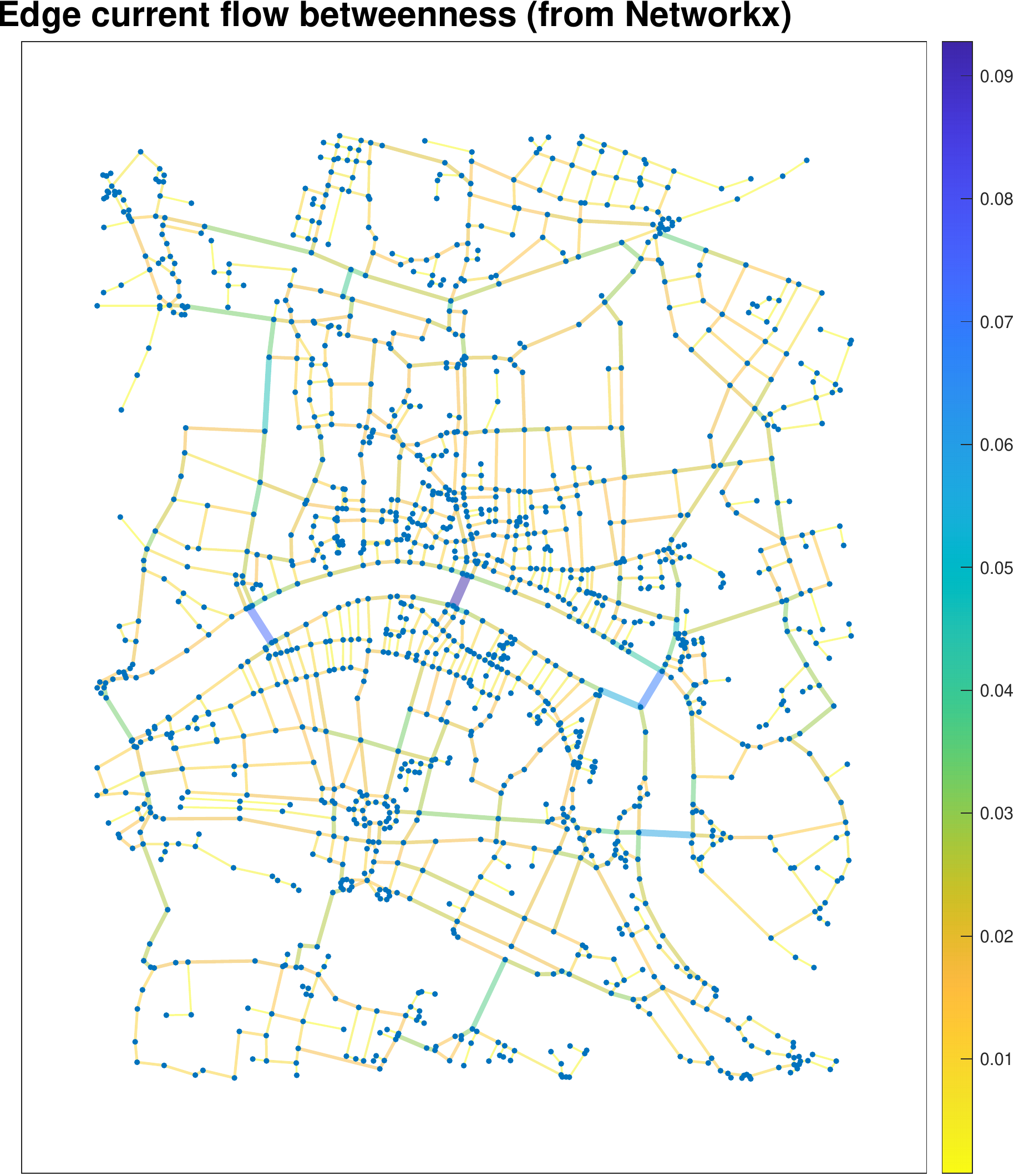}

\end{tabular}

\caption{Comparison of several centrality measures on a map of the Pisa city center; part II.} \label{fig:pisa2}
\end{figure}

A detailed comparison in terms of CPU time is difficult due to the very non-uniform state of the codebases for the algorithms that we have used and the different overheads in the two (interpreted) languages that we have used; nevertheless, we report the observed timings in Table~\ref{tab:timings}.
\begin{table}[]
    \centering
    \begin{tabular}{c|ccc}
    \toprule
        Algorithm & Matlab  & \begin{minipage}{2cm}Matlab  with\\ BGL~\cite{matlab_bgl} \end{minipage} & \begin{minipage}{2.8cm} Python 3.9.7 with\\ Networkx 2.4 \\ \cite{networkx} \end{minipage}\\
        \midrule
        Kemeny-based centrality & 0.40 s\\
        Pagerank & 0.001 s & & 0.19 s\\
        Pagerank dual & 0.001 s & & 0.1 s\\
        Betweenness dual & 0.04 s & 0.64 s & 6.21 s\\
        Edge betweenness & & 0.42 s & 6.51 s\\
        Edge current-flow betweenness & & & 7.3 s\\
        \bottomrule
    \end{tabular}
    \caption{CPU times for the various edge centrality algorithms on an Intel Core i5-1135G7 @ 2.40GHz laptop. The Matlab release is R2021a}
    \label{tab:timings}
\end{table}
In theoretical terms, the time complexity of the Kemeny-based centrality is comparable with the cost of one matrix inversion, or the solution of $n$ linear systems, which is $O(nm)$ when the Cholesky factor has $O(m)$ entries, as is the case for our road network examples. The complexity of edge current-flow betweenness is similar, as it is also based on the pseudoinverse of the Laplacian, while edge betweenness can be computed in $O(nm)$ as well for all networks. Measures computed using the dual graph have a similar cost because $m=O(n)$ for our road networks. Pagerank-based measures are significantly cheaper, as they require the solution of only one linear system instead of $n$.

Finally, we describe the results of an experiment at a much larger scale. We have used the same methodology to compute the Kemeny-based centrality of a road network of the Tuscany region, a very large graph with 1.22M nodes and 1.56M edges. Despite the large size of the network, the Cholesky factor $L$ is quite sparse, with only 3.36M nonzero entries, and it is computed in less than one second using Matlab. A much more challenging computation is the computation of the centralities of each edge, each one of which requires solving two triangular linear systems with $L$ and $L^T$. We have run this computation in parallel (using Matlab's \texttt{parfor}) on a machine with 12 physical cores with 3.4GHz speed each (Intel Xeon CPU E5-2643) and Matlab R2017a. The computation took 18 hours.

\section{Conclusions}\label{sec:con}
We have introduced a centrality measure for the edges of an undirected graph based on the variation of the Kemeny constant. This measure has been modified in order to avoid the Braess paradox. A regularization technique has been introduced for its computation; the technique allows one to detect cut-edges and to manage disconnected graphs. This Kemeny-based centrality can be expressed by means of the trace of suitable matrices, and its computation is ultimately reduced to the Cholesky factorization of a positive definite matrix, which is generally sparse. If the number of edges is huge, other techniques to estimate the trace of a matrix might be more appropriate, like the one proposed in \cite{ck} based on randomization. This is subject of further research.

%\bibliography{kementrality}
%\bibliographystyle{alphaurl}
%\bibliographystyle{siamplain}

\end{document}